\documentclass[11pt,final]{article}

\usepackage{amsmath}
\usepackage{amsthm}
\usepackage{enumitem}   
\usepackage{float}
\usepackage{graphicx}
\usepackage{grffile} 
\usepackage{mathrsfs}
\usepackage{mathtools}
\usepackage[ruled]{algorithm2e}
\usepackage{authblk}

\usepackage{geometry}
\usepackage{xcolor}

\usepackage[cmintegrals]{newpxmath}  

\usepackage{bm}                      

\usepackage{dsfont}                                              
  \newcommand{\IR}{\ensuremath\mathds{R}}                        
  
\newcommand*{\bftheta}{\vec{\theta}}

\newcommand*{\setE}{\ensuremath{\mathcal{T}}}                    
\newcommand*{\Gammah}{\Gamma_h}                                  
\renewcommand*{\vec}[1]{{\boldsymbol{#1}}}                       
\DeclareMathAlphabet{\mathbfsf}{\encodingdefault}{\sfdefault}{bx}{n}
\newcommand*{\vecc}[1]{\mathbfsf{#1}}                            
\newcommand*{\transpose}[1]{{#1}^\mathrm{T}}                     
\newcommand*{\normal}{\vec{n}}                                   
\newcommand*{\grad}{\vec{\nabla}}                                
\renewcommand*{\div}{\vec{\nabla}\cdot}                          
\newcommand*{\laplace}{\upDelta}                                 
\newcommand*{\llbrace}{\lbrace\hspace*{-0.18em}\vert}
\newcommand*{\rrbrace}{\vert\hspace*{-0.18em}\rbrace}
\newcommand*{\avg}[1]{\llbrace{#1}\rrbrace}                      
\newcommand*{\jump}[1]{\left\llbracket{#1}\right\rrbracket}      
\newcommand*{\abs}[1]{\ensuremath{|#1|}}                         
\newcommand*{\norm}[2]{\|#1\|_{#2}}                              
\newcommand*{\on}[2]{\left.#1\right\vert_{#2}}                   
\newcommand{\upwind}[1]{#1^{\uparrow}}                           

\newcommand*{\Rey}{\mathrm{Re}}                                  
\newcommand*{\Ca}{\mathrm{Ca}}                                   
\newcommand*{\Pe}{\mathrm{Pe}}                                   
\newcommand*{\Cn}{\mathrm{Cn}}                                   

\usepackage{caption}
\usepackage{longtable}
\usepackage{multirow}
\usepackage{tabularx}                                            
  \newcolumntype{R}{>{\raggedleft\arraybackslash}X}
  \newcolumntype{L}{>{\raggedright\arraybackslash}X}
  \newcolumntype{C}{>{\centering\arraybackslash}X}
\usepackage{booktabs}                                            
\usepackage{rotating}                                            
  
\usepackage{listings}
\newtheorem{theorem}{Theorem}
\newtheorem{lemma}{Lemma}

\newtheorem{remark}{Remark}
  
\usepackage{cleveref}                       
  \crefformat{equation}{Eq.~(#2#1#3)}
  \Crefformat{equation}{Equation~(#2#1#3)}
  \crefformat{figure}{Fig.~#2#1#3}          \crefname{figure}{Fig.}{Figs.}
  \Crefformat{figure}{Figure~#2#1#3}        \Crefname{figure}{Figure}{Figures}
  \crefformat{assumption}{Asm.~#2#1#3}
  \Crefformat{assumption}{Assumption~#2#1#3}
  \crefformat{lemma}{Lem.~#2#1#3}
  \Crefformat{lemma}{Lemma~#2#1#3}
  \crefformat{remark}{Rem.~#2#1#3}
  \Crefformat{remark}{Remark~#2#1#3}
  \crefformat{section}{Sec.~#2#1#3}         \crefname{section}{Sec.}{Secs.}
  \Crefformat{section}{Section~#2#1#3}      \Crefname{section}{Section}{Sections}
  \crefformat{table}{Tab.~#2#1#3}           \crefname{table}{Tab.}{Tables}
  \Crefformat{table}{Table~#2#1#3}          \Crefname{table}{Table}{Tables}
  \crefformat{theorem}{Thm.~#2#1#3}         \crefname{theorem}{Thm.}{Thms.}
  \Crefformat{theorem}{Theorem~#2#1#3}      \Crefname{theorem}{Theorem}{Theorems}

\usepackage[style=numeric-comp,%
            abbreviate=true,   
            useprefix=true,    %
            firstinits=true,   %
            hyperref=false,    %
            maxcitenames=2,    %
            maxbibnames=20,    %
            uniquename=init,   %
            sortcites=true,    
            doi=true,          %
            isbn=true,         %
            url=false,         %
            eprint=false,      %
            backend=bibtex     %
           ]{biblatex}         
  \renewbibmacro{in:}{}
  \DefineBibliographyStrings{english}{bibliography = {References}}

\title{A pressure-correction and bound-preserving discretization of the phase-field method for
variable density two-phase flows}
\author[a]{Chen Liu} 
\author[b]{Deep Ray} 
\author[a]{Christopher Thiele}
\author[a]{Lu Lin}
\author[a]{Beatrice Riviere\footnote{Corresponding author: Beatrice Riviere (riviere@rice.edu)}}
\affil[a]{Rice University, Department of Computational and Applied Mathematics, 6100 Main Street, Houston, TX 77005.}
\affil[b]{Department of Aerospace and Mechanical Engineering, University of Southern California, Los Angeles, CA, 90089.}
\date{\vspace{-0.3cm}October 29, 2020}

\begin{document}
\maketitle

\vspace{-0.75cm}\noindent\hrulefill\vspace{-0.5cm}
\section*{Abstract}
In this paper, we present an efficient numerical algorithm for solving the time-dependent Cahn--Hilliard--Navier--Stokes  equations that model the flow of two phases with different densities. The pressure-correction step in the projection method consists of a Poisson problem with a modified right-hand side. Spatial discretization is based on discontinuous Galerkin methods with piecewise linear or piecewise quadratic polynomials. Flux and slope limiting techniques successfully eliminate the bulk shift, overshoot and undershoot in the order parameter, which is shown to be bound-preserving. Several numerical results demonstrate that the proposed numerical algorithm is effective and robust for modeling two-component immiscible flows in porous structures and digital rocks. 
\\ \\
\emph{Keywords:} phase-field, pressure-correction projection, discontinuous Galerkin, flux limiters, slope limiters, Berea sandstone, digital rock
\par\noindent\hrulefill

%
\section{Introduction}
Phase-field models are popular mathematical models for multiphase problems, and they have been applied in fluid mechanics, hydro-geophysics and petroleum engineering. In this paper, we formulate a numerical method for solving the phase-field model for the flow of two phases with different densities at the pore scale. Because phase-field methods are based on the minimization of the Helmholtz free energy of the system, these methods have the advantage of implicitly tracking the interface and handling the motion of the contact line between the two phases and the solid boundary. Other popular methods for modeling two-phase flows at the pore scale include lattice Boltzmann methods and pore network models \cite{leclaire2017generalized,raeini2019validation}; see the surveys \cite{Kim2012,Shen2012} for pros and cons of these methods.  The underlying equations for phase-field methods are the coupled Cahn--Hilliard and Navier--Stokes (CHNS) equations, with unknowns order parameter, chemical potential, velocity and pressure. 
The study of the CHNS equations has recently received a lot of attention (see for instance \cite{Feng2006,bao2012finite,diegel2017convergence,LiuRiviere2018numericalCHNS} and references therein).  
\par
In this paper, we discretize the CHNS equations to  model the propagation of phases with different densities in
digital rock structures.  We propose a scheme that uses discontinuous Galerkin (DG) methods and that approximates the
unknowns with piecewise linear and piecewise quadratic polynomials in space. 
Velocity and pressure are decoupled according to the pressure-correction projection method. 
We refer to \cite{guermond2006overview} and the papers therein for an overview of the projection methods for the time-dependent incompressible Stokes equations. Because the density ratio of the phase densities is not equal to one, the pressure correction step requires the solution of an elliptic problem with  variable coefficients. This classical approach has two consequences on the stability and efficiency of the algorithm. First, the variable coefficient is a function of the order parameter and it may become negative if the numerical approximations of the order parameter are not bound-preserving. 
Second, the matrix assembly of the elliptic problem happens at each time step, which  can severely impact  the cost of the method for large size problems. Therefore in this work, following \cite{dong2012time,dodd2014fast}, we propose to replace the standard elliptic problem for the pressure correction step by a Poisson problem with a modified right-hand side.
This yields a stable system that only needs one matrix assembly throughout the whole simulation.
\par
The order parameter is a scalar field that can be viewed as a phase indicator since it takes a constant value in one bulk phase (phase $\mathrm{A}$) and another constant value in the other bulk phase (phase $\mathrm{B}$). If the order parameter is defined as the difference between mass fractions, it takes the value $+1$ in phase $\mathrm{A}$ and $-1$ in phase $\mathrm{B}$. The order parameter continuously increases from $-1$ to $+1$ across the diffuse interface between the two phases. Thus, in the context of incompressible immiscible two-phase flows, the physically admissible range of the order parameter is the interval $[-1,\,+1]$. However, the combination of using constant mobility and discontinuous piecewise polynomials of degree greater than or equal to one, produces numerical solutions that do not automatically satisfy a maximum principle in general. 
The phenomena of bulk shift, overshoot and undershoot have been observed for the advective Cahn--Hilliard equations \cite{LeeMunchSuli,FLAR2018finite};
they can be reduced by carefully selecting mesh sizes, time step values, and penalty values in the DG discrete forms. 
Thus, in this work, we propose two post-processing techniques to eliminate bulk shift, overshoot and undershoot with respect 
to the order parameter. First, we apply flux limiters that
produce a bound-preserving cell-averaged order parameter \cite{frank2019bound,kuzminmoller2005}. Second, we apply vertex-based slope limiting
techniques \cite{kuzmin2010vertex,kuzmin2013slope}.  We  highlight the differences between   the present work and \cite{LFTAR2018Efficient}, 
where we solved the phase-field model for two phases with equal densities. In addition to the different discretization of the
pressure-correction step, the discretization of the nonlinear reaction term in the momentum equation
also differs to take into account the varying density field and to produce a stable solution.  Finally, neither flux nor slope limiting techniques were employed in \cite{LFTAR2018Efficient} and the unknowns were approximated by discontinuous piecewise linears only.

\par
The outline of this paper is as follows. The mathematical model is introduced in~\Cref{sec:ModelProblem}, as well as boundary conditions that include a wettability model. The fully discrete numerical algorithm is defined in~\Cref{sec:numerical_scheme}. Several numerical examples are described in~\Cref{sec:numerical_experiments}, with concluding remarks in~\Cref{sec:conclusion}.

\section{Mathematical Model}\label{sec:ModelProblem}
We consider the flow of two immiscible phases in an open bounded polygonal domain $\Omega\!\subset\!\IR^3$ over the time interval $(0,T)$. The mathematical model for the CHNS system is written in a non-dimensional form and has four unknowns, namely the  order parameter $c$, chemical potential $\mu$, velocity $\vec{v}$, and pressure $p$.
\begin{subequations}\label{eq:CHNS:model_ndim}
\begin{align}
\partial_t c - \frac{1}{\Pe}\laplace{\mu} + \div{(c\vec{v})} &= 0 && \text{in}~(0,T)\times\Omega,\label{eq:CHNS:model_ndim:a}\\
\mu &= \Phi'(c) - \Cn^2\,\laplace{c} && \text{in}~(0,T)\times\Omega,\label{eq:CHNS:model_ndim:b}\\
\partial_t(\rho\vec{v}) \!+\! \div{(\rho\vec{v}\otimes\vec{v})} \!-\! \frac{1}{\Rey}\laplace{\vec{v}} &= -\frac{1}{\Rey\Ca}\grad{p} \!+\! \frac{3}{2\sqrt{2}\,\Rey\Ca\Cn}\mu\grad{c}\hspace{-0.2cm} && \text{in}~(0,T)\times\Omega,\label{eq:CHNS:model_ndim:c}\\
\div{\vec{v}} &= 0 && \text{in}~(0,T)\times\Omega.\label{eq:CHNS:model_ndim:d}
\end{align}
\end{subequations}
The dimensionless parameters $\Pe$, $\Cn$, $\Rey$, and $\Ca$ in \cref{eq:CHNS:model_ndim} are the P\'eclet number, Cahn number, Reynolds number, and capillary number, respectively. \Cref{eq:CHNS:model_ndim:a}-\eqref{eq:CHNS:model_ndim:b} form the advective
Cahn--Hilliard equations whereas \cref{eq:CHNS:model_ndim:c}-\cref{eq:CHNS:model_ndim:d} form the incompressible Navier--Stokes equations, which are coupled to the Cahn--Hilliard equations via the density $\rho$ that is a function of $c$ (see \cref{eq:CHNS:linear_mixing}), and the capillary stress forces $\mu \nabla c$. 
We employ the classical Ginzburg--Landau double well potential, which is written
as a sum of a convex function, $\Phi_+$, and a concave function, $\Phi_-$:
\begin{equation}\label{eq:CHNS:GL_potential}
\Phi(c) = \Phi_{+}(c) + \Phi_{-}(c), \quad
\Phi_{+}(c) = \frac{1}{4}(1+c^4), \quad
\Phi_{-}(c) = -\frac{1}{2}c^2.
\end{equation}
The (dimensionless) density of the bulk phase $\mathrm{A}$ (resp. $\mathrm{B}$) is a positive constant value $\rho_\mathrm{A}$ (resp. $\rho_\mathrm{B}$). Using the linear mixing rule, the density of the fluid varies continuously within the diffuse interface between the two bulk phases:
\begin{equation}\label{eq:CHNS:linear_mixing}
\rho(c) = \frac{1+c}{2}\rho_\mathrm{A} + \frac{1-c}{2}\rho_\mathrm{B}.
\end{equation}
Let $\normal$ be the unit outward  normal vector to the boundary $\partial\Omega$.  
The boundary of the computational domain is decomposed into three disjoint subsets, i.\,e., $\partial\Omega = \partial\Omega^{\mathrm{wall}}\cup\partial\Omega^{\mathrm{in}}\cup\partial\Omega^{\mathrm{out}}$. Here, 
$\partial\Omega^{\mathrm{wall}}$ denotes the solid boundary of the pore space, 
$\partial\Omega^{\mathrm{in}}$ denotes the inflow boundary, and $\partial\Omega^{\mathrm{out}}$ is the outflow boundary. 
\begin{equation*}
\partial\Omega^{\mathrm{in}} = \{\vec{x}\in\partial{\Omega}:\vec{v}\cdot\normal<0\}, \quad 
\partial\Omega^{\mathrm{out}} = \partial\Omega\setminus(\partial\Omega^{\mathrm{wall}}\cup\partial\Omega^{\mathrm{in}}).
\end{equation*}
The system \eqref{eq:CHNS:model_ndim} is completed by the following initial and boundary conditions:
\begin{subequations}\label{eq:CHNS:ICBC}
\begin{align}
c &= c^0 && \text{on}~\{0\}\times\Omega,\label{eq:CHNS:ICBC_1}\\
\vec{v} &= \vec{v}^0 && \text{on}~\{0\}\times\Omega,\label{eq:CHNS:ICBC_2}\\
c &= c_\mathrm{D} && \text{on}~(0,\,T)\times\partial\Omega^{\mathrm{in}},\label{eq:CHNS:ICBC_3}\\
\grad{c}\cdot\normal &= -\frac{\sqrt{2}\,\delta\, \cos(\theta)}{2\,\Cn}\,(c^2-1) && \text{on}~(0,\,T)\times(\partial\Omega^{\mathrm{wall}}\cup\partial\Omega^{\mathrm{out}}),
\label{eq:CHNS:ICBC_4}\\
\grad{\mu}\cdot\normal &= 0 && \text{on}~(0,\,T)\times\partial\Omega,\label{eq:CHNS:ICBC_5}\\
\vec{v} &= \vec{v}_\mathrm{D} && \text{on}~(0,\,T)\times\partial\Omega^{\mathrm{in}},\label{eq:CHNS:ICBC_6}\\
\vec{v} &= \vec{0} && \text{on}~(0,\,T)\times\partial\Omega^{\mathrm{wall}},\label{eq:CHNS:ICBC_7}\\
(\grad{\vec{v}} - \frac{1}{\Ca}p\vecc{I})\normal &= \vec{0} && \text{on}~(0,\,T)\times\partial\Omega^{\mathrm{out}}.\label{eq:CHNS:ICBC_8}
\end{align}
\end{subequations}
The values of the order parameter and velocity are prescribed on the inflow boundary by $c_\mathrm{D}$ and $\vec{v}_\mathrm{D}$ respectively. While these values can in general vary in time, we assume for  simplicity that they are independent of time in the remainder of the paper.
Wettability is modeled by a  user-specified 
contact angle $\theta$ that is enforced by the Neumann boundary condition \cref{eq:CHNS:ICBC_4}.  The input parameter $\delta$ is a 
scalar field that is equal to the constant one for smooth solid boundaries only and that otherwise corrects the numerical impact
of the jaggedness of the solid boundaries obtained from micro-CT scanning. The derivation of
this boundary condition and the wettability model can be found in \cite{FLSAR2017energy}.
\begin{remark}
It is well known that for any closed system ($\partial{\Omega}^{\mathrm{in}} = \emptyset$ and $\partial{\Omega}^{\mathrm{out}} = \emptyset$), the CHNS model \eqref{eq:CHNS:model_ndim} enjoys the global mass conservation property. Let $\vert \Omega\vert $ denote the
volume of $\Omega$.  We have the following identity for the order parameter:
\begin{equation*}
\frac{1}{\abs{\Omega}}\int_\Omega c(t) = \frac{1}{\abs{\Omega}}\int_\Omega c^0 = \bar{c}^0, \quad \forall t\in (0,T).
\end{equation*}
Furthermore, with the linear mixing rule \cref{eq:CHNS:linear_mixing}, we have the mass conservation identity
\begin{equation*}
\frac{1}{\abs{\Omega}} \int_\Omega \rho\big(c(t)\big) 
= \frac{1}{\abs{\Omega}} \int_\Omega \rho(c^0) 
= \frac{\rho_\mathrm{A}+\rho_\mathrm{B}}{2} + \frac{\rho_\mathrm{A}-\rho_\mathrm{B}}{2}\bar{c}^0,
\quad \forall t\in (0,T).
\end{equation*}
In \Cref{sec:numerical_experiments:spinodal} and \Cref{sec:numerical_experiments:merge_drop}, we verify this property for the discrete solution in closed systems.
\end{remark}

\section{Numerical Scheme}\label{sec:numerical_scheme}
In this section, we formulate a numerical method for solving \cref{eq:CHNS:model_ndim}-\cref{eq:CHNS:ICBC} that employs discontinuous Galerkin methods in space and operator splittings for the flow equations. 


\subsection{Time discretization}
Simulations of two-phase flows in digital rock require efficient numerical methods that are scalable on parallel clusters because the linear systems are very large, i.\,e., of the order $10^7-10^9$ unknowns. To decrease the size of the linear systems, the advective Cahn--Hilliard equations are decoupled from the Navier--Stokes equations. The pressure constraint is also decoupled from the incompressibility condition in the Navier--Stokes equations, according to the class of pressure projection methods, which are widely used for large scale computing.
\par
Uniformly partition the time interval $[0,\,T]$ into $N_T$ subintervals and let $\tau$ denote the time step length. 
For any $1 \leq n \leq N_T$, at each time step $t^n = n\tau$, we propose a semi-discrete in time scheme that consists of five steps:
\par
Step~1. Given $(c^{n-1},\vec{u}^{n-1})$, compute $(c^n,\mu^n)$ such that
\begin{subequations}\label{eq:CHNS:time_dis_const1}
\begin{align}
c^n \!- \frac{\tau}{\Pe}\laplace{\mu^n}\! + \tau\div(c^n\vec{u}^{n-1}) &= c^{n-1} && \text{in}~\Omega,\\
-\Cn^2\,\laplace{c^n} - \mu^n +\Phi_+\,\!'(c^n) &= -\,\Phi_-\,\!'(c^{n-1}) && \text{in}~\Omega,\\
c^n &= c_\mathrm{D} && \text{on}~\partial\Omega^{\mathrm{in}},\\
\grad{c^n}\cdot\normal &= -\frac{\sqrt{2}\, \delta \cos(\theta)}{2\,\Cn}\big((c^{n-1})^2-1\big)\!\!\! && \text{on}~\partial\Omega^{\mathrm{wall}}\cup\partial\Omega^{\mathrm{out}}\!,\\
\grad{\mu^n}\cdot\normal &= 0 && \text{on}~\partial\Omega.
\end{align}
\end{subequations}
\par
Step~2. Given $c^n$, compute $\rho^n$ 
\begin{equation}\label{eq:CHNS:time_dis_const2}
\rho^n = \frac{1+c^n}{2}\rho_\mathrm{A} + \frac{1-c^n}{2}\rho_\mathrm{B}.
\end{equation}
\par
Step~3. Given $(c^n, \mu^n, \rho^{n-1}, \rho^n, \vec{v}^{n-1}, p^{n-1}, \phi^{n-1})$, compute $\vec{v}^n$ such that
\begin{subequations}\label{eq:CHNS:time_dis_const3}
\begin{align}
\rho^n\vec{v}^n \!+\! \tau\div(\rho^n\vec{v}^{n-1}\otimes\vec{v}^n) - \frac{\tau}{\Rey}\,\laplace{\vec{v}^n} &= \rho^{n-1}\vec{v}^{n-1} 
\nonumber\\
-\frac{\tau}{\Rey\,\Ca}\,\grad{(p^{n-1}+\phi^{n-1})} 
&+ \frac{3\tau}{2\sqrt{2}\,\Rey\,\Ca\,\Cn}\mu^n\grad{c^n}\hspace{-0.35cm} && \text{in}~\Omega,\\
\vec{v}^n &= \vec{v}_\mathrm{D} && \text{on}~\partial\Omega^{\mathrm{in}},\\
\vec{v}^n &= \vec{0}  && \text{on}~\partial\Omega^{\mathrm{wall}},\\
(\grad{\vec{v}^n} - \frac{1}{\Ca}\,p^{n-1}\vecc{I})\normal &= \vec{0} && \text{on}~\partial\Omega^{\mathrm{out}}.
\end{align}
\end{subequations}
\par
Step~4. Given $(\rho^n, \vec{v}^n, \phi^{n-1})$, compute $\phi^n$ such that
\begin{subequations}\label{eq:CHNS:time_dis_const4}
\begin{align}
-\laplace{\phi^n} &= -\div{\Big((1-\frac{\rho_\mathrm{ref}}{\rho^{n}})\grad{\phi^{n-1}}\Big)} -\frac{\Rey\,\Ca\,\rho_\mathrm{ref}}{\tau}\,\div{\vec{v}^n} && \text{in}~\Omega,\\
\grad{\phi^n}\cdot\normal &= 0 && \text{on}~\partial\Omega^{\mathrm{in}}\cup\partial\Omega^{\mathrm{wall}},\\
\phi^n &= 0 && \text{on}~\partial\Omega^\mathrm{out}.
\end{align}
\end{subequations}
\par
Step~5. Given $(\vec{v}^{n}, p^{n-1}, \phi^{n-1}, \phi^n)$, compute $(p^n, \vec{u}^n)$ such that
\begin{subequations}\label{eq:CHNS:time_dis_const5}
\begin{align}
p^n &= p^{n-1} + \phi^n - \frac23 \Ca\,\div{\vec{v}^n},\\
\vec{u}^n &= \vec{v}^n - \frac{\tau}{\Rey\,\Ca\, \rho_\mathrm{ref}}\,\Big(\grad{\phi^n} + (\frac{\rho_\mathrm{ref}}{\rho^n} - 1)\grad{\phi^{n-1}}\Big).
\end{align}
\end{subequations}
In \cref{eq:CHNS:time_dis_const1}, we solve the advective Cahn--Hilliard equations by time-lagging the velocity. The convex part of the double well potential is evaluated implicitly in time whereas the concave part is evaluated at the previous time; this choice is known to yield a stable numerical solution of the order parameter \cite{eyre1998unconditionally}.  
In \cref{eq:CHNS:time_dis_const2}, we compute the density of the fluid, which varies in time and in space
as the location of the diffuse interface between the bulk phases changes. 
The \cref{eq:CHNS:time_dis_const3}-\cref{eq:CHNS:time_dis_const5} correspond to the temporal discretization of the incompressible Navier--Stokes equations, following a pressure-correction approach.  \Cref{eq:CHNS:time_dis_const3} solves for a velocity that satisfies the boundary conditions of the original problem. In \cref{eq:CHNS:time_dis_const4}, a potential function is obtained by solving a Poisson problem.  Finally \cref{eq:CHNS:time_dis_const5} updates the pressure and the divergence-free velocity field. 
The initial condition for the order parameter is $c^0$ (see \cref{eq:CHNS:ICBC_1}).  The algorithm needs additional initial conditions
because of the operator splittings. We choose $p^0 = 0$, $\phi^0 = 0, \rho^0 = \rho(c^0)$, and $\vec{u}^0 = \vec{v}^0$, where $\vec{v}^0$ is the initial velocity (see \cref{eq:CHNS:ICBC_2}).
\begin{remark}
Our formulation of \cref{eq:CHNS:time_dis_const4,eq:CHNS:time_dis_const5} is different from the standard pressure projection algorithm. Indeed, the standard method uses an elliptic problem with variable coefficient for the potential 
(namely $-\div{\big(\frac{1}{\rho^n}\grad{\phi^n}\big)}$) whereas our formulation uses a Poisson operator and modified right-hand side, which is based on the works \cite{dodd2014fast,dong2012time}. We select a reference density $\rho_\mathrm{ref}$ that is a positive constant. In our numerical results, we choose
\[
\rho_\mathrm{ref} = \min(\rho_\mathrm{A},\,\rho_\mathrm{B}).
\]
There are several advantages in using a Poisson problem in \cref{eq:CHNS:time_dis_const4}. 
The matrix in the linear system remains constant throughout the simulation and only the right-hand side of the linear system changes at each time step. This has significant impact on the computational cost as efficient solvers can be employed and the matrix (and preconditioner) assembly is performed at the first time step only.
For comparison, we present below Step 4 and 5 of the standard pressure projection algorithm:
\par
Step~4. (standard) Given $(\rho^n, \vec{v}^n)$, compute $\phi^n$ such that
\begin{subequations}\label{eq:CHNS:time_dis_standard4}
\begin{align}
-\div{\big(\frac{1}{\rho^n}\grad{\phi^n}\big)} &= -\frac{\Rey\,\Ca}{\tau}\,\div{\vec{v}^n} && \text{in}~\Omega,\label{eq:CHNS:time_dis_const4a}\\
\grad{\phi^n}\cdot\normal &= 0 && \text{on}~\partial\Omega^{\mathrm{in}}\cup\partial\Omega^{\mathrm{wall}},\label{eq:CHNS:time_dis_const4b}\\
\phi^n &= 0 && \text{on}~\partial\Omega^\mathrm{out}.\label{eq:CHNS:time_dis_const4c}
\end{align}
\end{subequations}
\par
Step~5. (standard) Given $(\vec{v}^{n}, p^{n-1}, \phi^n)$, compute $(p^n, \vec{u}^n)$ by the formula
\begin{subequations}\label{eq:CHNS:time_dis_standard5}
\begin{align}
p^n &= p^{n-1} + \phi^n - \frac{2}{3} \Ca\,\div{\vec{v}^n},\\
\vec{u}^n &= \vec{v}^n - \frac{\tau}{\Rey\,\Ca}\,\frac{1}{\rho^n}\grad{\phi^n}.
\end{align}
\end{subequations}
More importantly, coercivity of the discrete system for \cref{eq:CHNS:time_dis_const4} is always guaranteed, independently of the density $\rho^n$.  Because of the non-equal bulk densities and  \cref{eq:CHNS:time_dis_const2}, small overshoots or undershoots in the order parameter may produce non-physical negative densities, which will yield a loss of coercivity for the standard algorithm (see \cref{eq:CHNS:time_dis_standard4}). With the proposed approach, overshoots and undershoots of the
order parameter will not have an impact on the invertibility of the system.
\end{remark}
\begin{remark}
For the case of equal densities ($\rho_\mathrm{A} = \rho_\mathrm{B} = \rho_\mathrm{ref} = \rho^{0}$), 
\cref{eq:CHNS:time_dis_const4,eq:CHNS:time_dis_const5} reduce to the standard pressure-correction approach, namely \cref{eq:CHNS:time_dis_standard4,eq:CHNS:time_dis_standard5}. 
\end{remark}
\subsection{Fully discrete scheme}\label{sec:discretescheme}
Let $\setE_h = \{E_k\}$ be a partition of $\Omega$ where all the elements are cubes of the same size. The choice of cubic elements is well suited for micro-CT images of the rock because images of the pore space are themselves collections of (cubic) voxels. Let $h$ denote the maximum element diameter and let $\Gammah$ be the set of interior faces. For each interior face $e \in \Gammah$ shared by elements $E_{k^-}$ and $E_{k^+}$, with $k^-<k^+$, we define a unit normal vector $\normal_e$ that points from $E_{k^-}$ into $E_{k^+}$. For a boundary face, $e \subset \partial\Omega$, the normal vector $\normal_e$ is taken to be the unit outward vector to $\partial\Omega$. 
The scalar and vector unknowns belong to the spaces $X_h$ and $\vec{X}_h$ respectively. These discrete spaces  consist of discontinuous piecewise polynomials of degree $r \geq 1$:
\[
X_h =\{ \chi_h \in L^2(\Omega):\, \chi_h|_{E_k}\in\mathbb{Q}_r(E_k), \quad \forall E_k\in\setE_h\}, \quad
\vec{X}_h = (X_h)^3.
\]
The average $\avg{\chi}$ and jump $\jump{\chi}$ for any scalar function $\chi$ on boundary faces are defined to be its trace; and on interior faces they are defined by
\[
\avg{\chi}|_e = \frac{1}{2}\on{\chi}{E_{k^-}} + \frac{1}{2}\on{\chi}{E_{k^+}}, \quad
\jump{\chi}|_e = \on{\chi}{E_{k^-}} - \on{\chi}{E_{k^+}}, \quad 
\forall e = \partial E_{k^{-}}\cap\partial E_{k^{+}}.
\]
The $L^2$ inner-product on $\Omega$ (resp. on any face $e$) is denoted by $(\cdot,\cdot)$ (resp. $(\cdot,\cdot)_e$). We also make use of the following compact notation for the $L^2$ inner-product on the interior and boundary edges:
\[
(\cdot,\cdot)_{\mathcal{O}} = \sum_{e\in\mathcal{O}} (\cdot,\cdot)_e, \quad 
\mathcal{O} = \Gammah,~\partial\Omega,~\partial\Omega^\mathrm{in},~\partial \Omega^\mathrm{out}.
\]
The piecewise gradient (also called the broken gradient) is denoted by $\vec{\nabla}_h$. We now present the fully discrete scheme by first describing all the steps and then by defining the discrete forms. For each $n$, the scalar unknowns $c_h^n, \mu_h^n, \rho_h^n, p_h^n, \phi_h^n$ belong to the discrete space $X_h$ whereas the vector unknowns $\vec{u}_h^n, \vec{v}_h^n$ belong to $\vec{X}_h$.
\par
Input: The scalar functions $c_h^{n-1},\,\mu_h^n,\,\rho_h^{n-1},\,p_h^{n-1},\,\phi_h^{n-1}\,$ are given in $\,X_h\,$ and the 
\newline \hspace*{10.75ex} 
vector functions $\vec{u}_h^{n-1},\,\vec{v}_h^{n-1}$ are given in $\vec{X}_h$.
\par
Step~1. Compute $(c_h^n,\mu_h^n)$ such that for all $\chi_h\in X_h$
\begin{align}
(c_h^n,\chi_h) + \frac{\tau}{\Pe} a_\mathrm{diff}(\mu_h^n,\chi_h) &+ \tau a_\mathrm{adv}(c_h^n,\vec{u}_h^{n-1},\chi_h)\nonumber\\
&= (c_h^{n-1},\chi_h) - \tau(c_\mathrm{D}\, \vec{u}_h^{n-1}\cdot\vec{n_e},\chi_h)_{\partial\Omega^\mathrm{in}},\label{eq"disc1}\\
\Cn^2\,a_{\mathrm{diff},\partial\Omega^\mathrm{in}}(c_h^n,\chi_h) &- (\mu_h^n,\chi_h) +
(\Phi_+\,\!'(c_h^n),\chi_h) \nonumber\\
&= \Cn^2 b_\mathrm{diff}(c_h^{n-1};\chi_h) -(\Phi_-\,\!'(c_h^{n-1}),\chi_h).
\end{align}
%
\par
Step~2. Apply flux limiter and slope limiter (see \Cref{sec:numerical_scheme:flux} and \Cref{sec:numerical_scheme:slope}) to 
obtain post-processed order parameter, still denoted by $c_h^n$.
\begin{equation}
c_h^n \leftarrow \mathcal{S}(\mathcal{L}(c_h^{n-1},c_h^n,\mu_h^n,\vec{u}_h^{n-1})).
\end{equation}
\par
Step~3. Compute an updated chemical potential, still denoted by $\mu_h^n$ by solving
for all $\chi_h\in X_h$
\begin{align}
(\mu_h^n,\chi_h) 
&= \Cn^2\,a_{\mathrm{diff},\partial\Omega^\mathrm{in}}(c_h^n,\chi_h) + (\Phi_+\,\!'(c_h^n),\chi_h) \nonumber\\
&- \Cn^2 b_\mathrm{diff}(c_h^{n};\chi_h) + (\Phi_-\,\!'(c_h^{n}),\chi_h).
\end{align}
\par
Step~4. Compute $\rho_h^n$ 
\begin{equation}
\rho_h^n = \frac{1+c_h^n}{2}\rho_\mathrm{A} + \frac{1-c_h^n}{2}\rho_\mathrm{B}.
\end{equation}
\par
Step~5. Compute $\vec{v}_h^n$ such that for all $\bftheta_h\in\vec{X}_h$
\begin{align}
(\rho_h^n\vec{v}_h^n,\bftheta_h) + \tau\,a_\mathrm{reac}(\rho_h^n,\vec{v}_h^{n-1};\vec{v}_h^n,\bftheta_h)
+ \frac{\tau}{\Rey}\,a_\mathrm{ellip}(\vec{v}_h^n,\bftheta_h) = (\rho_h^{n-1}\vec{v}_h^{n-1},\bftheta_h)\nonumber\\
-\frac{\tau}{\Rey\,\Ca}\,b_\mathrm{pres}(p_h^{n-1},\phi_h^{n-1};\bftheta_h)
+ \frac{3\tau}{2\sqrt{2}\,\Rey\,\Ca\,\Cn} (\mu_h^n \, \vec{\nabla}_h c_h^n, \bftheta_h)
+\tau b_\mathrm{vel}(\bftheta_h).
\end{align}
\par
Step~6. Compute $\phi_h^n$ such that for all $\chi_h\in X_h$
\begin{align}
a_{\mathrm{diff},\partial\Omega^\mathrm{out}}(\phi_h^n,\chi_h) 
= b_\mathrm{dens}(\rho_h^n,\phi_h^{n-1};\chi_h)
-\frac{\Rey\,\Ca\,\rho_\mathrm{ref}}{\tau}\,(\vec{\nabla}_h \cdot \vec{v}_h^n, \chi_h).
\end{align}
\par
Step~7. Compute $(p_h^n, \vec{u}_h^n)$ such that for all $(\chi_h,\bftheta_h)\in X_h\times \vec{X}_h$
\begin{equation}
(p_h^n,\chi_h) = (p_h^{n-1},\chi_h) + (\phi_h^n,\chi_h) - \frac{2}{3} \Ca\,(\vec{\nabla}_h \cdot\vec{v}_h^n, \chi_h),
\end{equation}
\begin{equation}
(\vec{u}_h^n,\bftheta_h) + (\vec{\nabla}_h \cdot\vec{u}_h^n, \vec{\nabla}_h \cdot \bftheta_h) = 
(\vec{v}_h^n,\bftheta_h) 
- \frac{\tau}{\Rey\,\Ca \,\rho_\mathrm{ref}}\,(\vec{\nabla}_h\phi_h^n
+ (\frac{\rho_{\mathrm{ref}}}{\rho^n} - 1)\vec{\nabla}_h{\phi_h^{n-1}},\bftheta_h). \label{eq:discF}
\end{equation}
To start the algorithm, the initial discrete conditions are: $p_h^0=\phi_h^0=0$; $\vec{v}_h^0$ is the $L^2$ projection of $\vec{v}^0$; and $c_h^0$ is obtained by first applying the $L^2$ projection operator to $c^0$ and then by applying the slope limiter.  This will create an approximation $c_h^0$ that is bound-preserving. We point out that Step~1 yields a nonlinear system of equations, that will be solved by Newton's method. The order parameter is then post-processed by applying flux and slope limiters described in the next section. 
Step 5 yields a linear system of equations because the velocity is time-lagged in the nonlinear reaction term. We remark that for divergence-free velocity $\vec{v}$, we have
the identity:
\[
\div{(\rho \vec{v}\otimes \vec{v})} = \vec{v}\cdot\grad{(\rho\vec{v})}.
\]
Therefore, we propose the following DG discretization of the nonlinear reaction term:
\begin{multline*}
a_\mathrm{reac}(\rho,\vec{v};\vec{z},\bftheta) = 
(\vec{v}\cdot\vec{\nabla}_h(\rho \vec{z}),\bftheta) 
+ \frac{1}{2} (\vec{\nabla}_h \cdot\vec{v},\rho \vec{z}\cdot\bftheta)\\
- \frac{1}{2} (\jump{\vec{v}\cdot\normal_e}, \avg{\rho \vec{z} \cdot\bftheta})_{\Gammah\cup\partial\Omega^\mathrm{in}} 
+ \sum_{E\in\setE_h} \big(\abs{\avg{\vec{v}}\cdot\normal_{E}},((\rho\vec{z})^\mathrm{int} - (\rho \vec{z})^\mathrm{ext})\cdot\bftheta\big)_{\partial E_{-}^{\vec{v}}}.
\end{multline*}
Additional notation is needed for the definition of $a_\mathrm{reac}$. For an element $E\in\setE_h$, we denote by $\chi^\mathrm{int}$ (resp. $\chi^\mathrm{ext}$) the trace of the function $\chi$ on a side of $E$ coming from the interior (resp. exterior)
of $E$. Let $\vec{n}_E$ denote the unit outward normal vector to $E$. The upwind part of the boundary of $E$ with respect to $\vec{v}$ is denote by $\partial E_{-}^{\vec{v}}$. It is defined as
\[
\partial E_{-}^{\vec{v}} = \{ \vec{x}\in\partial E:~~ \avg{\vec{v}(\vec{x})}\cdot\vec{n}_E < 0\}.
\] 
The first term in $a_\mathrm{reac}$ is obtained by multiplying the term $\vec{v}\cdot\grad{(\rho\vec{v})}$ by a test function and integrating over all the mesh elements. The remaining terms in $a_\mathrm{reac}$ are added for numerical stability and they vanish if $\vec{z}=\vec{v}$ is the exact solution. In the case where $\rho_\mathrm{A}=\rho_\mathrm{B}$, the density $\rho$ reduces to one constant in the whole domain $\Omega$; and the form $a_\mathrm{reac}$ simplifies to a form that was introduced and analyzed in \cite{girault2005discontinuous}. 
\par 
Another new form that we propose in this paper is the discretization of the term 
$-\div{\Big((1-\frac{\rho_\mathrm{ref}}{\rho^{n}})\grad{\phi^{n-1}}\Big)}$ that appears in the right-hand side of
\cref{eq:CHNS:time_dis_const4}. 
\begin{align*}
b_\mathrm{dens}(\rho_h^n,\phi_h^{n-1};\chi_h)
=\, &((1-\frac{\rho_\mathrm{ref}}{\rho_h^n})\vec{\nabla}_h \phi_h^{n-1}\cdot\vec{\nabla}_h\chi_h) \\
-\, &(\llbracket (1-\frac{\rho_\mathrm{ref}}{\rho_h^n})(\vec{\nabla} \phi_h^{n-1}\cdot
\vec{n}_e) \chi_h\rrbracket, 1)_{\Gamma_h\cup\partial\Omega^\mathrm{out}}.
\end{align*}
The remaining forms are standard discretizations of the diffusion operator $-\laplace{\xi}$ and advection operator $\div{(\vec{v} \xi)}$ \cite{LFTAR2018Efficient,Rivierebook}. For completeness, we recall the forms below and we skip their derivation.
Let $\xi$ and $\eta$ be two scalar discrete functions.
\begin{align*}
a_\mathrm{diff}(\xi,\chi) &= 
(\vec{\nabla}_h \xi, \vec{\nabla}_h \chi)
-(\avg{\vec{\nabla}_h\xi\cdot\normal_e}, \jump{\chi})_{\Gammah}\\
&- (\avg{\vec{\nabla}_h\chi\cdot\normal_e}, \jump{\xi})_{\Gammah}
+\frac{\sigma}{h} (\jump{\xi},\jump{\chi})_{\Gammah},
\\
a_\mathrm{diff,in}(\xi,\chi) &= a_\mathrm{diff}(\xi,\chi)
- (\vec{\nabla}_h \xi\cdot\normal_e, \chi)_{\partial\Omega^\mathrm{in}}\\
&- (\vec{\nabla}_h \chi\cdot\normal_e, \xi)_{\partial\Omega^\mathrm{in}}
+\frac{\sigma}{h} (\xi, \chi)_{\partial\Omega^\mathrm{in}},
\\
a_\mathrm{diff,out}(\xi,\chi) &= a_\mathrm{diff}(\xi,\chi)
- (\vec{\nabla}_h\xi\cdot\normal_e, \chi)_{\partial\Omega^\mathrm{out}}\\
&- (\vec{\nabla}_h\chi\cdot\normal_e, \xi)_{\partial\Omega^\mathrm{out}}
+\frac{\sigma}{h} (\xi, \chi)_{\partial\Omega^\mathrm{out}},
\\
a_\mathrm{adv}(\xi,\vec{v},\chi) &=  
-(\xi,\vec{v} \cdot \vec{\nabla}_h \chi)
+ (\upwind{\xi} \avg{\vec{v}\cdot\normal_e}, \jump{\chi})_{\Gammah},
\end{align*}
where the upwind quantity $\upwind{\xi}$ on an interior face $e$ with normal $\vec{n}_e=\on{\vec{n}}{E_{k^-}}$ is defined by
\begin{align*}
\on{\xi^\uparrow}{e\in\Gammah} =
\begin{cases}
\on{\xi}{E_{k^-}} & \text{if}~\avg{\vec{v}}\cdot\vec{n}_e \geq 0, \\
\on{\xi}{E_{k^+}} & \text{if}~\avg{\vec{v}}\cdot\vec{n}_e < 0.
\end{cases} 
\end{align*}
The form $a_\mathrm{ellip}$ is the non-symmetric DG discretization of the vector differential operator $-\laplace{\vec{z}}$.  
\begin{multline*}
a_\mathrm{ellip}(\vec{z},\bftheta) = 
(\vec{\nabla}_h \vec{z}, \vec{\nabla}_h \bftheta)
- (\avg{(\vec{\nabla}_h\vec{z})\,\normal_e}, \jump{\bftheta})_{\Gammah}
+ (\avg{(\vec{\nabla}_h\bftheta)\,\normal_e}, \jump{\vec{z}})_{\Gammah} \\
+ \frac{\sigma}{h} (\jump{\vec{z}},\jump{\bftheta})_{\Gammah}
- ((\vec{\nabla}_h \vec{z})\,\normal_e, \bftheta)_{\partial\Omega^\mathrm{in}}
+ ((\vec{\nabla}_h \bftheta)\,\normal_e, \vec{z})_{\partial\Omega^\mathrm{in}}
+\frac{\sigma}{h} (\vec{z}, \bftheta)_{\partial\Omega^\mathrm{in}}.
\end{multline*}
The remaining forms in the right-hand sides of the discrete equations are
\begin{align*}
b_\mathrm{diff}(\xi;\chi) 
=& -\,(c_\mathrm{D}, \vec{\nabla}_h \chi\cdot\vec{n}_e)_{\partial\Omega^\mathrm{in}}
+ \frac{\sigma}{h} (c_\mathrm{D}, \chi)_{\partial\Omega^\mathrm{in}}\\
&- \frac{\sqrt{2}\delta \cos(\theta)}{2\,\mathrm{Cn}} (\xi^2-1,\chi)_{\partial\Omega^{\mathrm{wall}}\cup\partial\Omega^{\mathrm{out}}},\\
b_\mathrm{pres}(p,\phi;\bftheta) 
&= -\,(p, \vec{\nabla}_h \cdot \bftheta)
+(\avg{p}, \jump{\bftheta\cdot\vec{n}_e})_{\Gammah\cup\partial\Omega}
+(\vec{\nabla}_h \phi,\bftheta),\\
b_\mathrm{vel}(\bftheta)
&= -\,(\vec{v}_\mathrm{D}\cdot\vec{n}, \vec{v}_\mathrm{D} \cdot\bftheta)_{\partial\Omega^\mathrm{in}}
- \frac{1}{\Rey} (\vec{\nabla}_h \bftheta, \vec{v}_\mathrm{D}\cdot\vec{n}_e)_{\partial\Omega^\mathrm{in}}
+ \frac{\sigma}{h \Rey} (\vec{v}_\mathrm{D}, \bftheta)_{\partial\Omega^\mathrm{in}}.
\end{align*}
\par
We note that many of the forms above employ penalty parameters $\sigma>0$; they may take different values for different forms.  It is known that the value of the penalty parameter
has to be large enough to have coercivity of the forms  $a_\mathrm{diff}, a_\mathrm{diff,in}, a_\mathrm{diff,out}$ but it can simply
be taken equal to one for the coecivity of the form $a_\mathrm{ellip}$ \cite{Rivierebook}.  The different penalty values used in our
numerical results are made precise for each simulation in \Cref{sec:numerical_experiments}.


We end this section by stating important properties satisfied by the discrete order parameter, namely the bound preserving property
and the mass conservation property.

\begin{theorem}\label{thm:bdd}
The numerical approximation of the order parameter is bound preserving:
\begin{equation}\label{eq:bddch}
-1 \leq c_h^n(\vec{x}) \leq +1, \quad \vec{x}\in\Omega, \quad 0\leq n\leq N_T.
\end{equation}
\end{theorem}

\begin{theorem}\label{thm:mass}
For a closed system, i.e., in the case where $\partial\Omega^\mathrm{in} = \partial\Omega^\mathrm{out} = \emptyset$, we have
\begin{equation}\label{eq:mass1}
\frac{1}{\vert \Omega\vert} \int_\Omega c_h^n = \frac{1}{\vert \Omega\vert}\int_\Omega c_h^0, \quad \forall 0\leq n\leq N_T,
\end{equation}
and
\begin{equation}\label{eq:mass2}
\frac{1}{\vert \Omega\vert} \int_\Omega \rho(c_h^n) = \frac{1}{\vert \Omega\vert} \int_\Omega \rho(c_h^0), \quad \forall 0\leq n\leq N_T.
\end{equation}
\end{theorem}

The proofs of \Cref{thm:bdd} and \Cref{thm:mass} are given at the end of \Cref{sec:numerical_scheme:slope}.

\subsection{Element-wise mass average restriction}\label{sec:numerical_scheme:flux}
The order parameter takes the value $c_\ast = -1$ in one bulk phase and the value $c^\ast = +1$ in the other bulk phase. It is well known that the numerical approximation of the order parameter may exhibit a bulk shift in some parts of the domain, i.\,e., all the bulk values are either shifted up or down by a small amount \cite{FLAR2018finite,LeeMunchSuli}. The amount of bulk shift depends on the curvature of the interface, and it is reduced with decreasing mesh size.  In order to overcome this non-physical bulk shift, we propose to apply a flux limiting technique to the discrete order parameter that will produce a bound-preserving cell-averaged order parameter.  The flux limiter was recently introduced for discontinuous Galerkin discretizations of conservation laws in \cite{frank2019bound} and it is related to the class of flux-corrected transport algorithms (see \cite{kuzmin2012flux} and the references therein).
%
%
\par
We now describe the flux limiting technique that is applied to the discrete order parameter $c_h^{n}$ obtained at the end of Step~1 in the fully discrete scheme.  We assume that the order parameter at the previous time step satisfies the desired bounds:
\[
c_\ast \leq c_h^{n-1}(\vec{x}) \leq c^\ast, \quad \forall \vec{x}\in\Omega.
\]
The flux limiting approach consists of four successive steps:\\
\noindent{\bf Algorithm}: $c_{h}^n = \mathcal{L}(c_h^{n-1}, c_h^n, \mu_h^n, \vec{u}_h^{n-1})$ 
\begin{itemize}[leftmargin=1.5cm]
\item[Step~1.] Compute element-wise average at current and past time on each element $E \in \setE_h$. 
\begin{align*}
\bar{c}_{h,E}^{n-1} = \frac{1}{\abs{E}}\int_E c_h^{n-1} 
\quad \text{and} \quad
\bar{c}_{h,E}^{n} = \frac{1}{\abs{E}}\int_E c_h^{n}.  
\end{align*}
\item[Step~2.] Fix an element $E\in\setE_h$ and compute the flux, $\mathcal{H}_E(e)$,  on each face $e\subset\partial E$. We recall that $\vec{n}_E$ is the unit outward normal vector to $E$.
\begin{align*}
\forall e \in \Gammah,\, e = \partial{E} \cap \partial{E'}, &&
\mathcal{H}_{E}(e) &= 
- \frac{1}{\Pe}\int_{e}\avg{\grad{\mu_h^n}}\cdot\normal_{E} 
+ \int_{e} (c_h^n)^\uparrow \avg{\vec{u}_h^{n-1}}\cdot\normal_{E},\\
&& &+ \frac{\sigma}{\Pe\,h}\int_{e}(\on{\mu_h^n}{E} - \on{\mu_h^n}{E'}),\\
\forall e \subset\partial{E}\cap\partial{\Omega}^\mathrm{in}, &&
\mathcal{H}_{E}(e) &= \int_{e} c_\mathrm{D} \, \vec{u}_h^{n-1}\cdot\normal_{E},\\
\forall e\subset\partial{E}\cap\partial{\Omega}^\mathrm{wall}, &&
\mathcal{H}_{E}(e) &= 0,\\
\forall e\subset\partial{E}\cap\partial{\Omega}^\mathrm{out}, &&
\mathcal{H}_{E}(e) &= \int_{e} \upwind{(c_h^n)} \vec{u}_h^{n-1}\cdot\normal_{E}.
\end{align*}
The function $\mathcal{H}_{E}(e)$ measures the net mass flux across each face of the element $E$, into a neighboring element if $e$ is an interior face or into the exterior of the computational domain if $e$ is a boundary face.
\item[Step~3.] On each element $E \in \setE_h$, apply an iterative algorithm to limit the fluxes. 
\[
\forall E\in\setE_h, \quad \bar{c}_{h,E}^\mathrm{post} = \mathcal{L}_a(E,\bar{c}_{h,E}^{n-1},\mathcal{H}_E).
\]
\item[Step~4.] Update the order parameter as follows:
\begin{equation*}
\forall E\in\setE_h, \quad \on{c_{h}^{n}}{E} \leftarrow 
\on{c_{h}^{n}}{E} - \bar{c}_{h,E}^{n} + \bar{c}^\mathrm{post}_{h,E}.
\end{equation*}
\end{itemize}
It remains to explain the iterative algorithm used to limit the fluxes.
%
\par
\noindent{\bf Algorithm}: $\bar{c}_{h,E}^\mathrm{post} = \mathcal{L}_a(E,\bar{c}_{h,E}^{n-1},\mathcal{H}_E)$ 
\begin{itemize}[leftmargin=1.75cm]
\item[Step~1.] Initialization: $m=1, \bar{c}^{(0)}_{h,E} = \bar{c}_{h,E}^{n-1}$, $\mathcal{H}^{(0)}_E = \mathcal{H}_E, \, 
\mathcal{H}_E^{(-1)} = 0$ and 
$\alpha_E(e) = 1$ for all $e\in\partial E$.
\item[Step~2.]
Compute the following quantities:
\begin{subequations}\label{eq:flux_algorithm:PQ}
\begin{align}
P^{+}_E &= \tau\!\sum_{e\subset\partial E} \max{\!\big(0,\, -\mathcal{H}_{E}^{(m-1)}(e)\big)}, &
Q^{+}_E &= \abs{E}\big(c^\ast - \bar{c}^{(m-1)}_{h,E}\big),\label{eq:flux_algorithm:PQ1}\\
P^{-}_{E} &= \tau\!\sum_{e\subset\partial E} \min{\!\big(0,\, -\mathcal{H}_{E}^{(m-1)}(e)\big)}, &
Q^{-}_{E} &= \abs{E}\big(c_\ast - \bar{c}^{(m-1)}_{h,E}\big).\label{eq:flux_algorithm:PQ2}
\end{align}
\end{subequations}
\item[Step~3.] Compute limiting factors $\alpha_{E}(e)$ for all faces $e\subset\partial E$.\\
If $e$ is an interior face such that $e\subset\partial E\cap\partial E'$:
\begin{itemize}[leftmargin=0.5cm, label=\labelitemi]
\item If $\mathcal{H}_{E}^{(m-1)}(e) < 0$, then set $\alpha_{E}(e) = \min(\alpha^{+}_E,\, \alpha^{-}_{E'})$, \\
where $\alpha^{+}_E = \min(1,\, Q^{+}_E / P^{+}_E)$ and $\alpha^{-}_{E'} = \min(1,\, Q^{-}_{E'} / P^{-}_{E'})$.
\item If $\mathcal{H}_{E}^{(m-1)}(e) > 0$, then set $\alpha_{E}(e) = \min(\alpha^{-}_E,\, \alpha^{+}_{E'})$, \\
where $\alpha^{-}_E = \min(1,\, Q^{-}_E / P^{-}_E)$ and $\alpha^{+}_{E'} = \min(1,\, Q^{+}_{E'} / P^{+}_{E'})$.
\end{itemize}
If $e$ is a boundary face such that $e\subset\partial E\cap\partial \Omega$:
\begin{itemize}[leftmargin=0.5cm, label=\labelitemi]
\item If $\mathcal{H}_{E}^{(m-1)}(e) < 0$, then set $\alpha_{E}(e) = \min(1,\, Q^{+}_E / P^{+}_E)$.
\item If $\mathcal{H}_{E}^{(m-1)}(e) > 0$, then set $\alpha_{E}(e) = \min(1,\, Q^{-}_E / P^{-}_E)$. 
\end{itemize}
\item[Step~4.] Update $\bar{c}^{(m)}_{h,E}$ and $\mathcal{H}^{(m)}_E$ by:
\begin{subequations}
\begin{align}
\bar{c}^{(m)}_{h,E} &= \bar{c}^{(m-1)}_{h,E} - \frac{\tau}{\abs{E}} \sum_{e\subset\partial E} \alpha_{E}(e) \mathcal{H}_{E}^{(m-1)}(e), \label{eq:flux_algorithm:step4_1}\\
\mathcal{H}_{E}^{(m)}(e) &= \big(1 - \alpha_{E}(e)\big)\,\mathcal{H}_{E}^{(m-1)}(e), \quad\forall e\subset\partial{E}.\label{eq:flux_algorithm:step4_2}
\end{align}
\end{subequations}
\item[Step~5.] 
If
$\displaystyle \max_{\forall E\in\setE_h}\max_{\forall e\subset\partial{E}}\abs{\mathcal{H}_{E}^{(m)}(e)} < \epsilon_1$ or
$\displaystyle \max_{\forall E\in\setE_h}\max_{\forall e\subset\partial{E}}\abs{\mathcal{H}_{E}^{(m)}(e) - \mathcal{H}_{E}^{(m-1)}(e)} < \epsilon_2$, \\
\hspace*{0.5cm} set $\bar{c}^\mathrm{post}_{h,E} = \bar{c}^{(m)}_{h,E}$, \\
Else \\
\hspace*{0.5cm} set $m \leftarrow m+1$ and go to Step~2.
\end{itemize} 
The solution obtained by the flux limiting algorithm has the following boundedness property.
\begin{lemma}\label{thm:flux:boundedness}
Let $E$ be a mesh element and let $\{ \bar{c}_{h,E}^{(k)}\}_k$ be the sequence obtained in the iterative algorithm $\mathcal{L}_a$.
Assume that the iterate $\bar{c}_{h,E}^{(m-1)}$ belongs to the interval $[c_\ast,c^\ast]$. Then the next iterate
$\bar{c}_{h,E}^{(m)}$ also belongs to the interval $[c_\ast,c^\ast]$.
\end{lemma}
\begin{proof}
Let us check the upper bound: $\bar{c}^{(m)}_{h,E} \leq c^\ast$. Since the
iterate $\bar{c}^{(m-1)}_{h,E}$ belongs to the interval $[c_\ast,\,c^\ast]$, it is easy to see that 
$\alpha_{E}(e)\geq0$ for all $ e\subset\partial{E}$.
We apply the inequality $\xi \leq \max(0,\xi)$ to \cref{eq:flux_algorithm:step4_1} and use the definition of $\alpha_{E}(e)$. 
We obtain
\begin{align*}
\bar{c}^{(m)}_{h,E} 
&\leq \bar{c}^{(m-1)}_{h,E} + \frac{\tau}{\abs{E}} \sum_{e\subset\partial E} \alpha_{E}(e) \max\Big(0, -\mathcal{H}_{E}^{(m-1)}(e)\Big)\\
&= \bar{c}^{(m-1)}_{h,E} + \frac{\tau}{\abs{E}} \sum_{e\subset\partial E} \min(\alpha^{+}_E,\, \alpha^{-}_{E'}) \max\Big(0, -\mathcal{H}_{E}^{(m-1)}(e)\Big)\\
&\leq \bar{c}^{(m-1)}_{h,E} + \frac{\tau}{\abs{E}} \sum_{e\subset\partial E} \frac{Q^{+}_E }{P^{+}_E} \max\Big(0, -\mathcal{H}_{E}^{(m-1)}(e)\Big).
\end{align*}
Therefore, with the definition of $P_E^+$ and $Q_E^+$, we have
\begin{equation*}
\bar{c}^{(m)}_{h,E} 
\leq \bar{c}^{(m-1)}_{h,E} + c^\ast - \bar{c}^{(m-1)}_{h,E}
= c^\ast.
\end{equation*}
The proof for the lower bound $\bar{c}^{(m)}_{h,E} \geq c_\ast$ follows a similar argument.
\end{proof}
The next result states that the algorithm $\mathcal{L}_a$ converges.
\begin{lemma}\label{thm:flux:convergence}
The sequence $\{ \mathcal{H}_{E}^{(m)}\}_{m,E}$ defined in the iterative algorithm $\mathcal{L}_a$ converges uniformly
over all elements $E$. For any $\epsilon>0$, there is $M$ such that
\[ 
\max_{\forall E\in\setE_h}\max_{\forall e\subset\partial{E}}\abs{\mathcal{H}_{E}^{(m)}(e)} < \epsilon 
~~\text{or}~~
\max_{\forall E\in\setE_h}\max_{\forall e\subset\partial{E}}\abs{\mathcal{H}_{E}^{(m)}(e) - \mathcal{H}_{E}^{(m-1)}(e)} 
< \epsilon, \quad \forall m\geq M.
\]
\end{lemma}
\begin{proof}
By induction on $m$, it is easy to show, with the previous lemma, that
\[
\max_{\forall E\in\setE_h}\max_{\forall e\subset\partial{E}}\abs{\mathcal{H}_{E}^{(m)}(e)}
\leq \max_{\forall E\in\setE_h}\max_{\forall e\subset\partial{E}}\abs{\mathcal{H}_{E}^{(m-1)}(e)}.
\]
Therefore, convergence is immediately obtained.
\end{proof}
As an immediate corollary, the output of the algorithm $\mathcal{L}_a$ is bound-preserving.
\begin{lemma}\label{lem:boundpres}
Assume that the cell-average, $\bar{c}_{h,E}^{n-1}$, of the discrete order parameter at time $t^{n-1}$, 
belongs to the interval $[c_\ast,c^\ast]$ for all elements $E\in\mathcal{T}_h$. Fix a mesh
element $E$ and define $\bar{c}_{h,E}^\mathrm{post} = \mathcal{L}_a(E,\bar{c}_{h,E}^{n-1},\mathcal{H}_E)$. 
Then, we have
\begin{equation}
c_\ast \leq \bar{c}_{h,E}^\mathrm{post} \leq c^\ast, \quad \forall E\in \mathcal{T}_h.
\end{equation}
\end{lemma}
%
\subsection{Slope limiting post-processing}\label{sec:numerical_scheme:slope}
The flux limiter described in the previous section ensures that the element-wise average of the order parameter attains values that belong to the interval $[c_\ast,c^\ast]$. Let us first consider the case of piecewise linear approximations ($r=1$). 
In that case, using the barycenter of each element $E$, we can write the discrete solution as:
\begin{equation}\label{eq:Taylorform}
c_h^n(\vec{x}) = \bar{c}_h^n + \grad{c_h^n}(\vec{x}_\mathrm{c})\cdot(\vec{x}-\vec{x}_\mathrm{c}),\quad \forall \vec{x} \in E.
\end{equation}
While the constant part is bound-preserving ($\bar{c}_h^n \in [c_\ast,c^\ast]$), the linear part of the solution may violate
the physical bounds. This lack of maximum principle for DG approximations is refered to as overshoot (resp. undershoot) 
if the upper (resp. lower) bound is not preserved.  We apply a slope limiter technique to eliminate the overshoot and undershoot
phenomena in the order parameter.
%
The algorithm for limiting the order parameter $c_h^n$ at each time step $t^n$ is as follows:\\
\noindent{\bf Algorithm}: $c_{h}^n = \mathcal{S}(c_h^n)$ 
\begin{itemize}[leftmargin=1.75cm]
\item[Step~1.] Detect troubled cells: if there exists a point $\vec{x}$ inside an element $E$ such that $c_h^n(\vec{x})$ is outside the
interval $[c_\ast,c^\ast]$, then mark this element $E$ as a troubled cell and go to Step~2, otherwise move to the next mesh element.  
\item[Step~2.] Using the form \eqref{eq:Taylorform}, find the largest slope correction factor $\beta_E \in [0,\,1]$\,, such that after post-processing, the limited order parameter field
\begin{equation*}
c_h^n(\vec{x}) ~=~ \bar{c}_h^n + \beta_E\,\grad{c_h^n}(\vec{x}_\mathrm{c})\cdot(\vec{x}-\vec{x}_\mathrm{c})
\end{equation*}
is bounded below by $c_\ast$ and above by $c^\ast$ on all vertices of element $E$. The procedure for finding $\beta_E$ is introduced in \cite{kuzmin2010vertex,kuzmin2013slope}.
\end{itemize}
\par
Next, we consider the case where the numerical approximation of the order parameter is a higher degree polynomial, namely $r\geq 2$. 
If one element $E$ is marked as a troubled cell, the discrete solution is reduced to a linear polynomial (for instance by an $L^2$ projection) on that element $E$. Then, Step~2 is applied.\\

\emph{Proof of \Cref{thm:bdd}}:\\
We first remark that the initial approximation $c_h^0$ is obtained by an $L^2$ projection of $c^0$, followed by
the slope limiter described in this section.  Therefore, the cell average $\bar{c}_{h,E}^0$ belongs to $[-1,1]$ for
all mesh elements $E$ 
and the limited $c_h^0$ is such that $-1\leq c_h^0(\vec{x}) \leq +1$ for all $\vec{x}\in E$ and for all elements $E$.

The theorem is then obtained by induction on $n$.  Assume that $\bar{c}_{h,E}^{n-1}$ belongs to $[-1,1]$ for all mesh elements $E$.
\Cref{lem:boundpres} implies that after the application of the flux limiter, the cell average is bound-preserving:
\[
-1 \leq \bar{c}_{h,E}^n\leq +1, \quad \forall E\in\mathcal{T}_h.
\]
In the second step, the slope limiter is applied for the mesh elements where the bounds are violated at a given point in the element. 
For linear polynomials, the extrema will occur at the vertices and after application of the slope limiter, these extrema will belong to $[-1,1]$.
Clearly this implies \eqref{eq:bddch}. For polynomials of degree greater than one, the approximation on a given troubled cell is reduced
to a linear polynomial. Therefore the bounds still hold. $\hfill\square$\\

\emph{Proof of \Cref{thm:mass}}:\\
Fix $n\geq 1$. We observe that
\[
\int_\Omega c_h^n = \sum_{E\in\mathcal{T}_h} \vert E \vert \, \bar{c}_{h,E}^n.
\]
Since the slope limiter preserves the cell average, it suffices to study the effect of the flux limiter.
We have after the application of the flux limiter: $\bar{c}_{h,E}^n = \bar{c}_{h,E}^\mathrm{post}$
and $ \bar{c}_{h,E}^\mathrm{post}$ is obtained when convergence of the iterative process of algorithm $\mathcal{L}_a$
is reached for a given tolerance. In other words, there is $M > 0$ such that $\bar{c}_{h,E}^\mathrm{post} = \bar{c}_{h,E}^{(M)}$.

Let us next prove by induction on $m$ that the following two identities hold 
for any $e=\partial E\cap\partial E'$:
\begin{subequations}\label{eq:flux:mass_conservation}
\begin{align}
\alpha_{E}(e) &= \alpha_{E'}(e),\label{eq:flux:mass_conservation1}\\
\mathcal{H}_{E}^{(m)}(e) &= -\mathcal{H}_{E'}^{(m)}(e).\label{eq:flux:mass_conservation2}
\end{align}
\end{subequations}
These identities certainly hold true for $m=0$ from Step 1 of the algorithm $\mathcal{L}_a$.
Assume \cref{eq:flux:mass_conservation} is true for $m-1$. Then, at the next iteration step $m$, we have the following three cases.
\begin{itemize}[leftmargin=0.75cm]
\item[1.] If $\mathcal{H}_{E}^{(m-1)}(e) < 0$, namely $\mathcal{H}_{E'}^{(m-1)}(e) > 0$, then
\begin{align*}
\alpha_{E}(e) &= \min{\Big(\min{\Big(1,\frac{Q_{E}^+}{P_{E}^+}\Big)},\,\min{\Big(1,\frac{Q_{E'}^-}{P_{E'}^-}\Big)}\Big)} = \min{\Big(1,\frac{Q_{E}^+}{P_{E}^+},\frac{Q_{E'}^-}{P_{E'}^-}\Big)},\\
\alpha_{E'}(e) &= \min{\Big(\min{\Big(1,\frac{Q_{E'}^-}{P_{E'}^-}\Big)},\,\min{\Big(1,\frac{Q_{E}^+}{P_{E}^+}\Big)}\Big)} = \min{\Big(1,\frac{Q_{E}^+}{P_{E}^+},\frac{Q_{E'}^-}{P_{E'}^-}\Big)}.
\end{align*}
\item[2.] If $\mathcal{H}_{E}^{(m-1)}(e) = 0$, namely $\mathcal{H}_{E'}^{(m-1)}(e) = 0$, then the factors $\alpha_{E}(e)$ and $\alpha_{E'}(e)$ are not updated.
\item[3.] If $\mathcal{H}_{E}^{(m-1)}(e) > 0$, namely $\mathcal{H}_{E'}^{(m-1)}(e) < 0$, then
\begin{align*}
\alpha_{E}(e) &= \min{\Big(\min{\Big(1,\frac{Q_{E}^-}{P_{E}^-}\Big)},\,\min{\Big(1,\frac{Q_{E'}^+}{P_{E'}^+}\Big)}\Big)} = \min{\Big(1,\frac{Q_{E}^-}{P_{E}^-},\frac{Q_{E'}^+}{P_{E'}^+}\Big)},\\
\alpha_{E'}(e) &= \min{\Big(\min{\Big(1,\frac{Q_{E'}^+}{P_{E'}^+}\Big)},\,\min{\Big(1,\frac{Q_{E}^-}{P_{E}^-}\Big)}\Big)} = \min{\Big(1,\frac{Q_{E}^-}{P_{E}^-},\frac{Q_{E'}^+}{P_{E'}^+}\Big)}.
\end{align*}
\end{itemize}
Therefore, the identity \cref{eq:flux:mass_conservation1} holds for $m$. 
\Cref{eq:flux:mass_conservation2} is immediately obtained by substituting \cref{eq:flux:mass_conservation1} 
into \cref{eq:flux_algorithm:step4_2}.
\begin{equation*}
\mathcal{H}_{E}^{(m)}(e) 
= \big(1 - \alpha_{E}(e)\big)\,\mathcal{H}_{E}^{(m-1)}(e)
= -\big(1 - \alpha_{E'}(e)\big)\,\mathcal{H}_{E'}^{(m-1)}(e) 
= -\mathcal{H}_{E'}^{(m)}(e).
\end{equation*}
Next, for any $m$, with \cref{eq:flux:mass_conservation} and with the fact that $\mathcal{H}_E^{(m-1)}(e) = 0$ for a boundary face, we
have
\begin{equation*}
\sum_{E\in\setE_h} \sum_{e\subset\partial E} \alpha_{E}(e) \mathcal{H}_{E}^{(m-1)}(e) = 
\sum_{\begin{array}{ll}e\in\Gamma_h\\e=\partial E\cap\partial E'\end{array}} \alpha_E(e) \left(\mathcal{H}_{E}^{(m-1)}(e)+\mathcal{H}_{E'}^{(m-1)}(e)\right) = 0.
\end{equation*} 
With the above equality, \Cref{eq:flux_algorithm:step4_1} implies 
\begin{equation*}
\sum_{E\in\setE_h}\vert E \vert \, \bar{c}_{h,E}^{(m)} = \sum_{E\in\setE_h}\vert E \vert \, \bar{c}_{h,E}^{(m-1)}, \quad \forall m,
\end{equation*}
which yields 
\begin{equation*}
\sum_{E\in\setE_h}\vert E \vert \, \bar{c}_{h,E}^{\mathrm{post}} = \sum_{E\in\setE_h}\vert E \vert \, \bar{c}_{h,E}^{(0)}
= \sum_{E\in\setE_h}\vert E \vert \, \bar{c}_{h,E}^{n-1}.
\end{equation*}
Since $\bar{c}_{h,E}^{n} = \bar{c}_{h,E}^{\mathrm{post}}$, we have proved \eqref{eq:mass1}.
The conservation law of mass density \eqref{eq:mass2} follows by applying the linear mixing rule \cref{eq:CHNS:linear_mixing}. $\hfill\square$

\section{Numerical Experiments}\label{sec:numerical_experiments}
In this section, we apply the proposed numerical algorithm to both closed and open systems, which include a spinodal decomposition, two merging droplets, flows in micro structure, and flows in Berea sandstone. We also report the computational performance and scalability results. 
\subsection{Spinodal decomposition}\label{sec:numerical_experiments:spinodal}
The spinodal decomposition is a widely used benchmark problem for modeling the transition of two phases from a thermodynamical unstable initial condition to an equilibrium condition. The system is closed $\partial\Omega = \partial\Omega^\mathrm{wall}$ and throughout the evolution of the decomposition, the global mass is preserved.
\par
The computational domain is a toroidal shape pipe, with the following definition:
\begin{equation*}
\Omega = \left\{(x,y,z)\in(0,\,1)^3:~ \left(\sqrt{(x-0.5)^2 + (y-0.5)^2} - 0.35\right)^2 + (z-0.5)^2 < 0.15^2\right\}.
\end{equation*}
The time step size is $\tau = 10^{-3}$ and the mesh resolution (the edge length of cubic elements) is $h_e=1/128$.
The initial velocity field is taken to be $\vec{v}^{0} = \vec{0}$. The initial order parameter field is generated by sampling numbers from a discrete uniform distribution, as follows:
\begin{equation*}
\on{c^{0}}{E_{k}} \sim \mathcal{U}{\{-1,+1\}}.
\end{equation*}
We will show the impact of the flux and slope limiters as well as the impact of the wettability of the solid wall. For all simulations,  we  set the same initial order parameter $c^0$. Piecewise linear approximations are used and the other parameters for these simulations are:
\[
\rho_\mathrm{A} = 1200,\quad \rho_\mathrm{B} = 800,\quad \Rey = 1,\quad \Ca = 10^{-1},\quad \Pe = 1,\quad \Cn = h_e.
\]
The penalty values for the forms are: $\sigma=2$ for $a_\mathrm{diff}$ and $\sigma = 8$ for $a_\mathrm{ellip}$.  
In the remainder of the paper, tolerances $\epsilon_1 = \epsilon_2 = 10^{-7}$ are chosen for the flux limiting step.
%
%
\par
\Cref{fig:numerical_experiments:SD_1} displays snapshots of the order parameter field at different time steps and for two different values of contact angle.  We refer to phase $\mathrm{A}$ the bulk phase with order parameter $c_h^n = +1$ and phase $\mathrm{B}$ the bulk phase with order parameter $c_h^n=-1$. The center of the interface, i.\,e. the set of points for which $c_h^n=0$, is displayed in green; phase $\mathrm{A}$ is in red and phase $\mathrm{B}$ is in transparent blue. The meaning of these colors is fixed throughout the rest of the paper.
The case $\theta = 90^\circ$ corresponds to a neutral wall for both phases whereas the case $\theta = 180^\circ$ corresponds to a super-hydrophobic wall for phase $\mathrm{A}$. The top row of snapshots in \Cref{fig:numerical_experiments:SD_1} shows that in the case of neutral wall, each of the two phases occupies three disjoint sections of the domain. The interfaces are perpendicular to the solid wall. The bottom row shows that the mixture evolution is different for the super-hydrophobic wall. In this case, phase $\mathrm{A}$ is completely repelled from the walls and occupies  a donut-like shape region of the domain. The results shown in \Cref{fig:numerical_experiments:SD_1} were obtained with our proposed numerical scheme, that includes flux and slope limiting.
\par
To show the effects of these limiting techniques, we now compare the approximations of the order parameter with the limiting turned off.
This means that Step 2 of the discrete scheme in \Cref{sec:discretescheme} is skipped.  The solution is extracted from the middle plane $\{(x,y,z)\in(0,\,1)^3:~ z=0.5\}$ and it is displayed in \Cref{fig:numerical_experiments:SD_2} at different time steps.
We employ a rainbow color scale which maps the values in $[-1,\,+1]$ from blue to red and we employ the color black for values of the order parameter that are outside the interval $[-1,\,+1]$.  In other words, the black regions are the regions where the discrete order parameter is not bound-preserving and exhibits overshoot and undershoot. The first row of snapshots corresponds to the neutral wall case $(\theta=90^\circ)$ with limiting whereas the second row corresponds to the neutral wall case without limiting.  We observe that the dynamics of the decomposition are similar with or without limiting; however, overshoots and undershoots are eliminated when flux and slope limiting are used. Similar conclusions can be made for the case of super-hydrophobic walls (third and fourth rows).  The third row shows the solution with our scheme whereas the fourth row shows the solution without limiting.
\par
Finally, since the system is closed, we can numerically verify that the mass conservation property is satisfied. 
\Cref{fig:numerical_experiments:SD_3} displays the average of the order parameter, $\bar{c}_h^n$, and the average of the density, $\bar{\rho}_h^n$.  We observe that the mass is globally conserved throughout the simulations, for the case of neutral and super-hydrophobic walls, with or without limiting.
\begin{figure}[ht!]
\centering
\includegraphics[width=\textwidth]{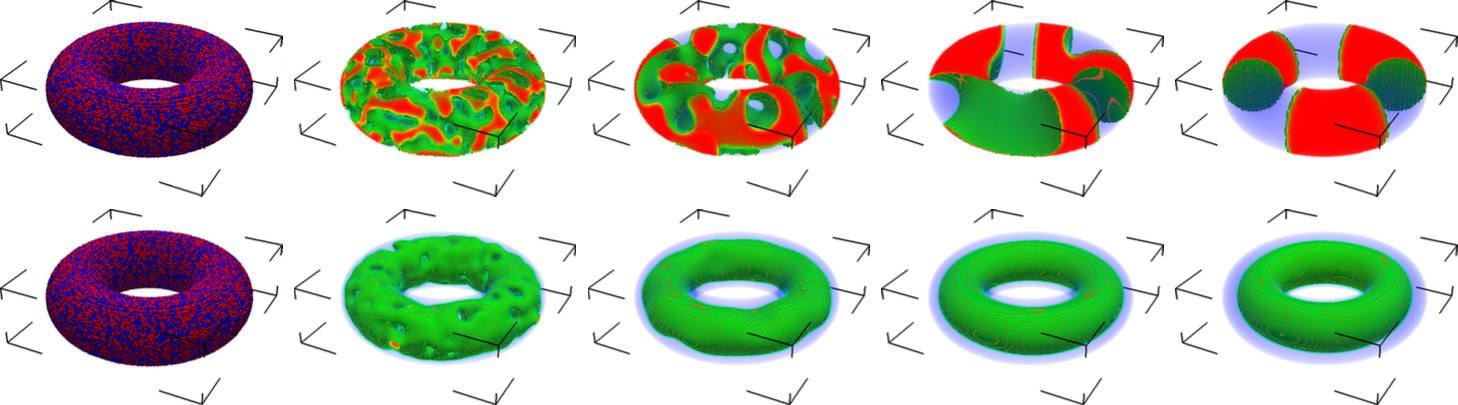}
\caption{3D views of the evolution of the order parameter field. Selected snapshots for initial data and simulation results at time step $2^{4}$, $2^{7}$, $2^{10}$, and $2^{13}$. The center of the diffusive interface ($c_h^n=0$) is colored in green. The top row corresponds
to neutral wall (contact angle $\theta = 90^\circ$)  and the bottom row to super-hydrophobic wall ($\theta = 180^\circ$).}
\label{fig:numerical_experiments:SD_1}
\end{figure}
\begin{figure}[ht!]
\centering
\includegraphics[width=\textwidth]{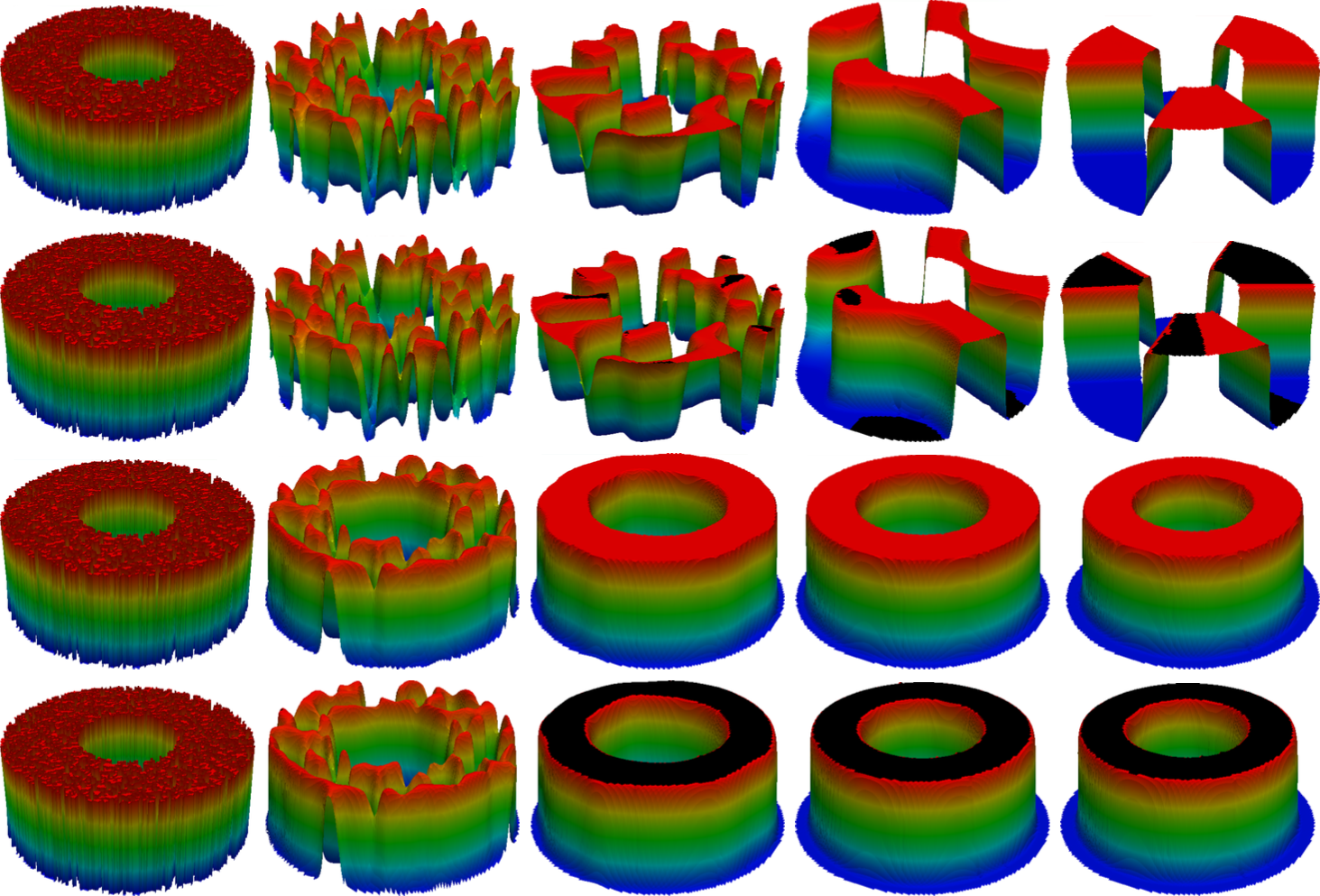}
\caption{Plots of the order parameter field extracted from the plane $\{z=0.5\}$. Selected snapshots for initial data and simulation results at time step $2^{4}$, $2^{7}$, $2^{10}$, and $2^{13}$. Values ourside the interval $[-1,\,+1]$ are displayed in black. From top row to bottom row: contact angle $\theta=90^\circ$ with our numerical scheme; contact angle $\theta=90^\circ$ without limiters; contact angle $\theta=180^\circ$ with our numerical scheme; contact angle $\theta=180^\circ$ without limiters.}
\label{fig:numerical_experiments:SD_2}
\end{figure}
\begin{figure}[ht!]
\centering
\includegraphics[width=\textwidth]{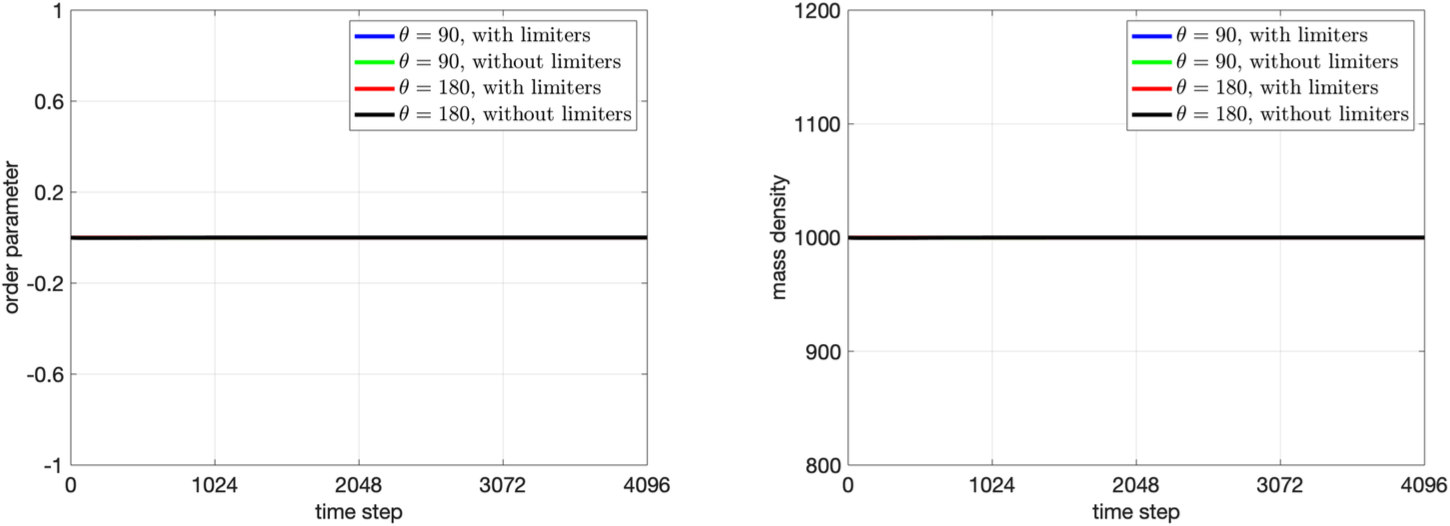}
\caption{Spinodal decomposition example. Left figure: average of order parameter as a function of time. Right figure: average of mass density as a function of time.}
\label{fig:numerical_experiments:SD_3}
\end{figure}

\subsection{Merging droplets}\label{sec:numerical_experiments:merge_drop}
For this second example, the computational domain is the unit cube $\Omega = (0,1)^3$ and the system is closed, $\partial\Omega = \partial\Omega^\mathrm{wall}$. Two droplets of phase $\mathrm{A}$ are initially in a non-equilibrium configuration, surrounded by phase $\mathrm{B}$, and as  time evolves, they merge
into one larger droplet.
During this process, the large droplet wobbles several times and eventually the two droplets evolve into the most thermodynamically favorable configuration, namely, a single spherical droplet \cite{bao2012finite,magaletti2013sharp}. In addition, as a closed system, mass is conserved throughout the whole dynamic evolution.
\par
The initial velocity field is taken to be $\vec{v}^0 = \vec{0}$. Meanwhile, the initial order parameter field is prescribed by the following formula:
\begin{align*}
c^{0}(x,y,z) = 
\max\Big\{-1,\,
&\tanh{\Big(\frac{0.25 - \norm{\transpose{[x, y, z]}-\transpose{[0.35, 0.35, 0.35]}}{}}{\sqrt{2}\,\Cn}\Big)},\\
&\tanh{\Big(\frac{0.25 - \norm{\transpose{[x, y, z]}-\transpose{[0.65, 0.65, 0.65]}}{}}{\sqrt{2}\,\Cn}\Big)}\Big\},
\end{align*}
where $\norm{\cdot}{}$ denotes the Euclidian norm. 
\par 
The discrete space is the space of piecewise quadratic polynomials ($r=2$), the mesh resolution is $h_e=1/64$ and the time step is $\tau = 10^{-4}$. The wall is assumed to be neutral ($\theta=90^\circ$). 
The penalty values for the forms are: $\sigma=4$ for $a_\mathrm{diff}$ and $\sigma = 32$ for $a_\mathrm{ellip}$. For the other parameters, we choose
\[
\rho_\mathrm{A} = 1200,\quad \rho_\mathrm{B} = 800,\quad \Rey = 1,\quad \Ca = 10^{-4},\quad \Pe = 1,\quad \Cn = h_e.
\]
\par
\Cref{fig:numerical_experiments:drop_1} displays snapshots of the order parameter field as well as its  value along the diagonal of the computational domain, i.\,e., along the line
$\{(x,y,z) \in (0,\,1)^3:~ x = y = z\}$.
The snapshots clearly show the merging of the two droplets, the intermediate wobbling stages and finally the equilibrium configuration of the spherical droplet. The plots of the order parameter along the diagonal of the domain, show that the overshoot and undershoot phenomena are not present. This is because of the flux and slope limiting used in our numerical method. Finally, the global mass is conserved, as shown in \Cref{fig:numerical_experiments:drop_2}.
\begin{figure}[ht!]
\centering
\includegraphics[width=\textwidth]{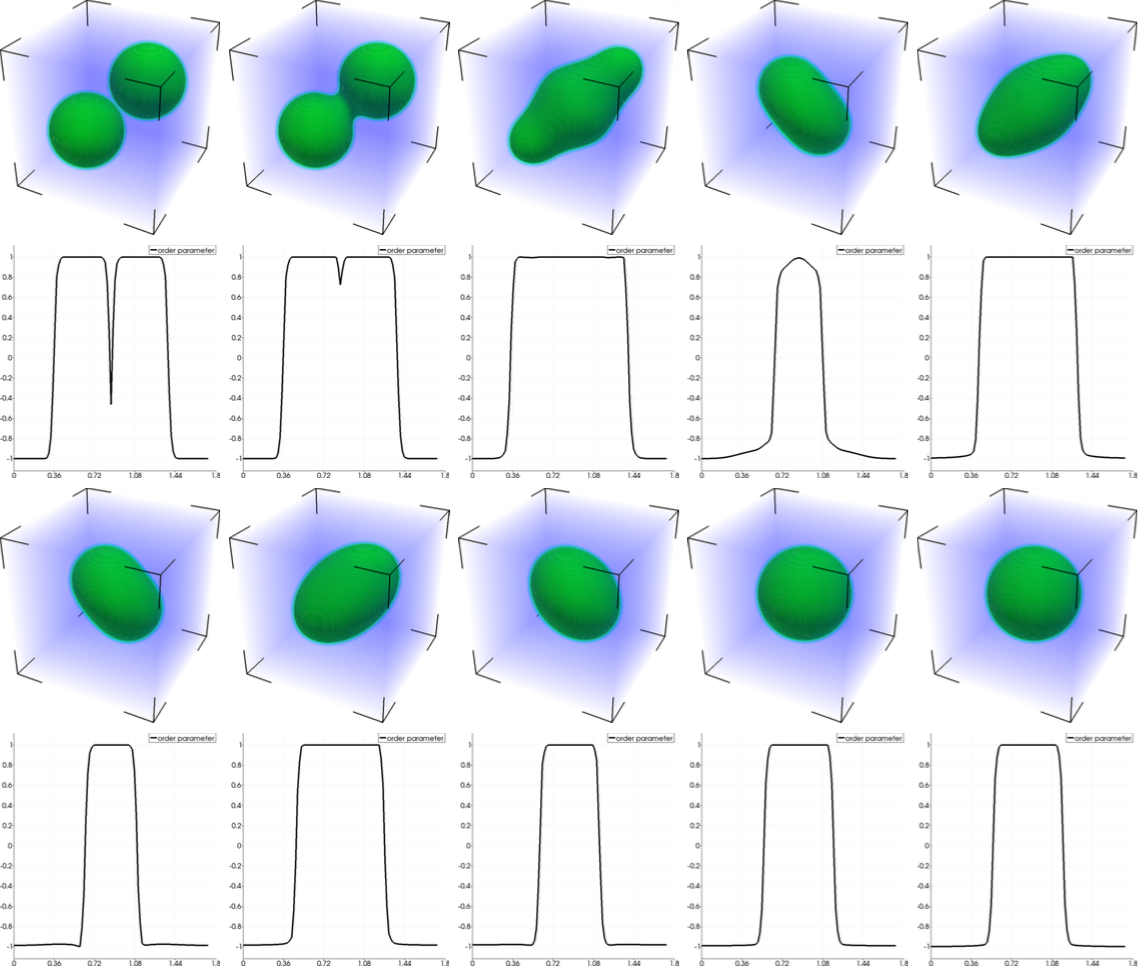}
\caption{3D views of the evolution of order parameter field (first and third rows) and plots of order parameter extracted along the line $\{x = y = z\}$ (second and fourth rows). Selected snapshots at time step $0$, $4$, $22$, $46$, $78$, $112$, $144$, $180$, $500$, and $1000$. The center of the diffusive interface ($c_h^n=0$) is colored in green.}
\label{fig:numerical_experiments:drop_1}
\end{figure}
\begin{figure}[ht!]
\centering
\includegraphics[width=\textwidth]{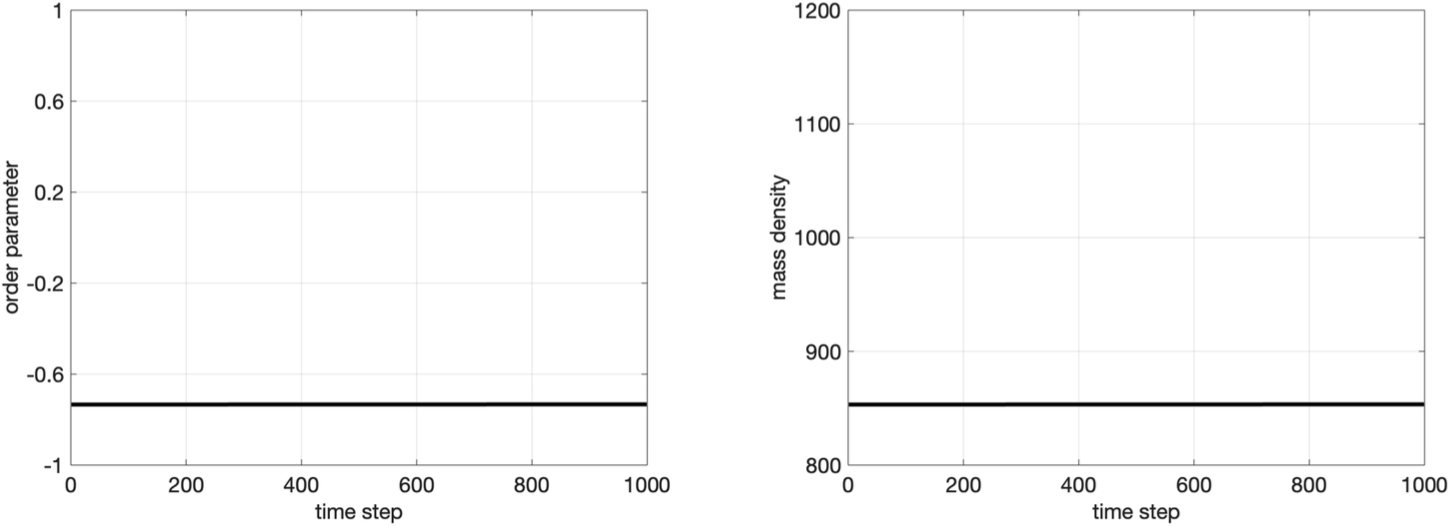}
\caption{Merging droplets example. Left figure: average of order parameter as a function of time. Right figure: average of mass density as a function of time.}
\label{fig:numerical_experiments:drop_2}
\end{figure}

\subsection{Micro structure simulations}
Micro structures are engineered porous media that are commonly used in microfluidic devices. The micro structures are made of connected cavities that are repeated according to a given pattern, see \Cref{fig:numerical_experiments:micro_geometry}. 
Because they are good substitutes for real porous media, micro structures are widely used by scientists to analyze complex behavior of fluid dynamics at the pore scale.
\begin{figure}[ht!]
\centering
\includegraphics[width=0.38\textwidth]{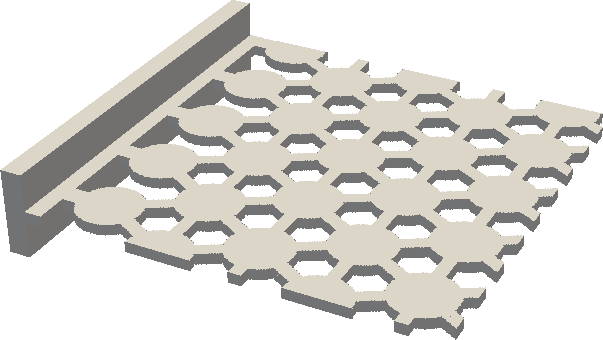}
\caption{The micro structure domain attached to a buffer region on the left (inflow boundary).}
\label{fig:numerical_experiments:micro_geometry}
\end{figure}
\par
In this numerical example, we study the flow of two phases in the micro structure given in \Cref{fig:numerical_experiments:micro_geometry}. The figure shows the microstructure ($380\times 400 \times 10$ cubic elements) and a buffer region added to the left side of the microstructure.
In this open system, the inflow boundary is the face $\{x=0\}$ of the buffer region and the outflow boundary is the right side of the micro structure. 
The pore space is initially filled with phase $\mathrm{B}$.  The initial velocity field is taken to be $\vec{v}^0 = \vec{0}$ and the velocity on inflow boundary is defined by
\begin{equation*}
\vec{v}_\mathrm{D}(x,y,z) = 400 y\left(y-1\right)\left(z-\frac{2}{5}\right)\left(z-\frac{3}{5}\right), \quad (x,y,z)\in\partial\Omega^\mathrm{in}.
\end{equation*} 
\par
The discrete space is the space of piecewise linears ($r=1$), the mesh resolution is $h_e = 1/400$ and the time step size is $\tau = 5\times 10^{-4}$. The wall is hydrophobic with respect to phase $\mathrm{A}$ (contact angle $\theta=120^\circ$). 
The other parameters for the simulations are
\[
\rho_\mathrm{A} = 1200,\quad \rho_\mathrm{B} = 800,\quad \Rey = 1,\quad \Ca = 1,\quad \Pe = 50,\quad \Cn = h_e.
\]
The penalty parameters are: 
$\sigma = 4$ for the forms $a_\mathrm{diff}, a_\mathrm{diff,\partial\Omega^\mathrm{in}}$, 
$\sigma = 8$ for all interior and outflow faces for the form $a_\mathrm{diff,\partial\Omega^\mathrm{out}}$, 
$\sigma = 8$ (resp. $32$) for all interior (resp. inflow boundary) faces of $a_\mathrm{ellip}$, 
and $\sigma = 32$ for the form $b_\mathrm{vel}$.  
\par
\Cref{fig:micro_2dview} shows the evolution of the order parameter $c_h^n$ along the plane $\{z=0.5\}$.  We observe that phase $\mathrm{A}$ invades the microstructure while staying away from the solid walls because of the wettability constraint.  
\begin{figure}[ht!]
\centering
\includegraphics[width=\textwidth]{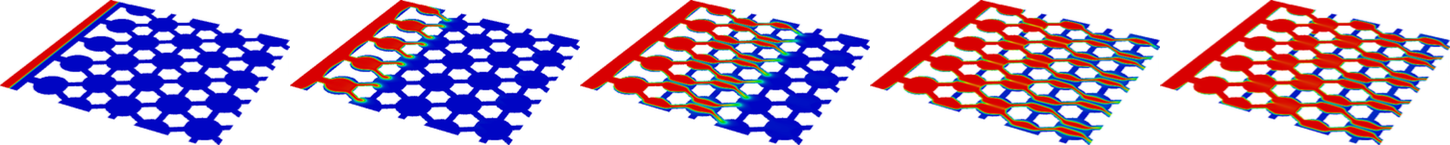}
\caption{Contours of the order parameter along the plane $\{z=0.5\}$ at different time steps:
$0$, $50$, $100$, $150$, and $200$.}
\label{fig:micro_2dview}
\end{figure}
\par
\Cref{fig:numerical_experiments:micro_1} compares the plots of the order parameter $c_h^n(x,y,0.5)$ obtained with and without flux/slope limiting. The top row corresponds to our numerical method whereas the bottom row corresponds to the case of no limiting. The regions where the order parameter violates the bounds $[-1,+1]$ are shown in black. We observe that the flux and slope limiters remove any overshoot
and undershoot phenomena. Note that the dynamics are similar for both cases. 

%
%
\begin{figure}[ht!]
\centering
\includegraphics[width=\textwidth]{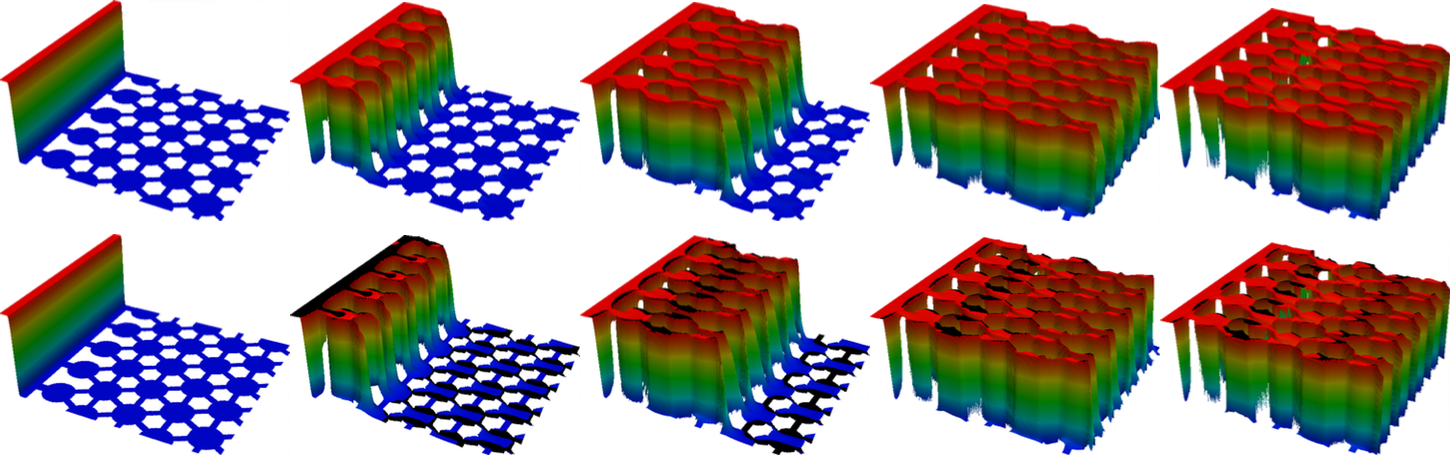}
\caption{Plots of the order parameter $c_h^n(x,y,0.5)$ with limiting (top row) and without limiting (bottom row) for different time steps: $0$, $50$, $100$, $150$, and $200$. The values outside the interval $[-1,\,+1]$ are marked in black.}
\label{fig:numerical_experiments:micro_1}
\end{figure}

\subsection{Berea rock simulations}\label{sec:numerical_experiments:porous}
Digital rock technology employs 3D images of porous rock, that are constructed by X-ray based micro-CT scans of rock samples.
The pore space image is a set of cubic voxels that directly form the computational domain for our numerical method.
\begin{figure}[ht!]
\centering
\includegraphics[width=0.38\textwidth]{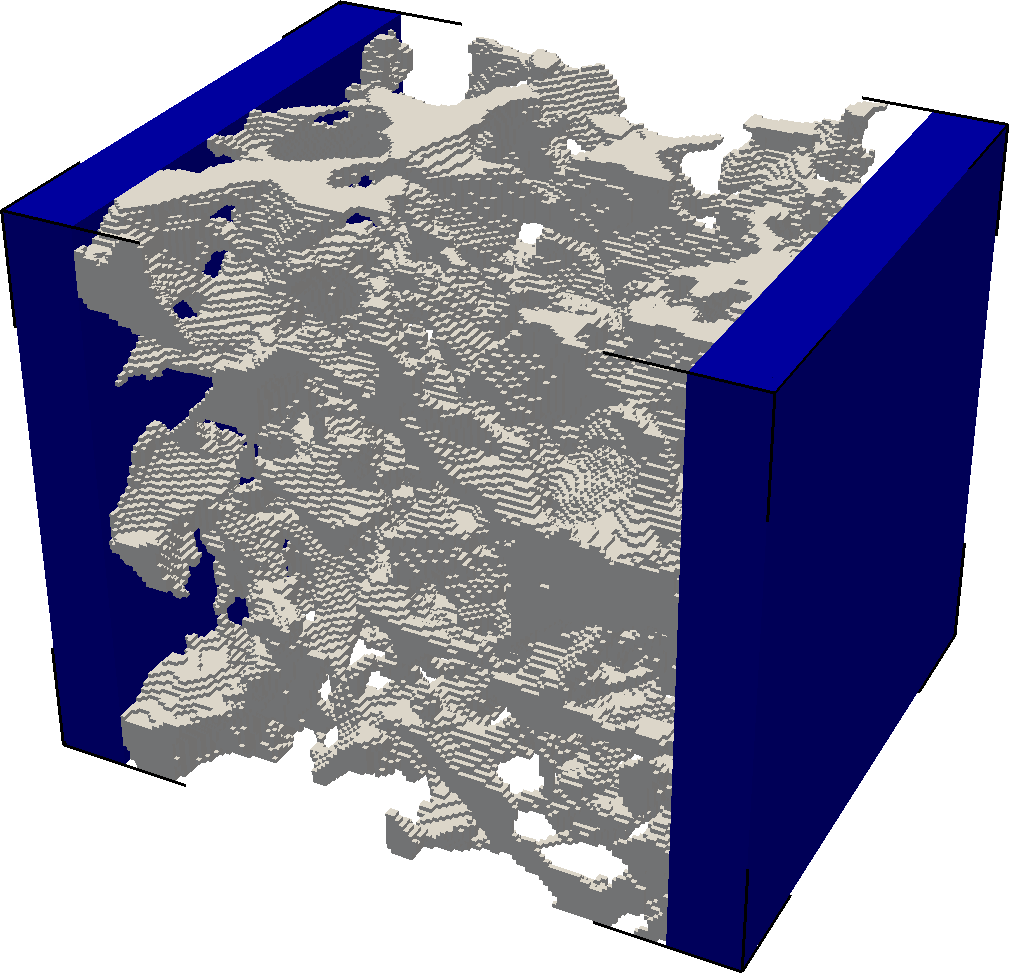} 
\caption{Computational domain for Berea simulations: the gray region is the pore space of the rock sample, blue regions are the buffers at the inlet (left buffer) and outlet (right buffer) boundaries.}
\label{fig:numerical_experiments:berea_1}
\end{figure}
\par
In this example, the pore space is a three-dimensional image of a Berea sandstone with in-situ phase $\mathrm{B}$ (oil). Buffer zones are added at the inlet and outlet faces of the rock sample, see \Cref{fig:numerical_experiments:berea_1}, as is the case in lab experiments. 
The buffer zone at the outlet is also filled with phase $\mathrm{B}$ whereas half of the buffer zone at the inlet is filled with phase $\mathrm{A}$ (water). 
The initial velocity field is taken to be $\vec{v}^0 = \vec{0}$ and the velocity on inflow boundary is defined by 
\begin{equation*}
\vec{v}_\mathrm{D}(x,y,z) 
= \frac{1}{10}\left(x-\frac{1}{10}\right)\left(x-\frac{9}{10}\right)\left(z-\frac{1}{10}\right)\left(z-\frac{9}{10}\right), \quad (x,y,z)\in \partial\Omega^\mathrm{in}.
\end{equation*} 
The space of discontinuous piecewise linears is used and the mesh resolution is $h_e=1/160$. The penalty values are $\sigma=2$ for the interior faces of $a_\mathrm{diff}, a_\mathrm{diff,in}$; $\sigma=100$ on the inflow bounday of $a_\mathrm{diff,in}$; $\sigma=8$ for the form $a_\mathrm{diff,out}$; $\sigma=8$ (resp. $32$) for the interior (resp. boundary) faces of $a_\mathrm{ellip}$ and finally $\sigma = 32$ for $b_\mathrm{vel}$.  We choose for contact angle $\theta = 80^\circ$, which means that the solid faces of the pores are hydrophylic with respect to phase $\mathrm{A}$.  We vary the capillary number 
$\Ca \in \{1, 10^{-1}, \cdots, 10^{-3}\}$ to study the effect of capillary forces on the displacement of the phases.
The other parameters are:
\[
\rho_\mathrm{A} = 1000,\quad \rho_\mathrm{B} = 800,\quad \Rey = 1,\quad \Pe = 1,\quad \Cn = h_e.
\]
\par
\Cref{fig:numerical_experiments:berea_2} displays the order parameter field after the injection of $10$ pore volume (PV), which corresponds to the time $t=3.6$. 
We observe that there are small differences in the propagation of phase $\mathrm{A}$ into the pore space for large capillary numbers, namely for $\Ca = 1$ and $\Ca = 10^{-1}$. For these cases, the dominant forces are the viscous forces.  As $\Ca$ decreases, the phase distribution in the pore space changes drastically, in particular for $\Ca = 10^{-2}$ and $\Ca = 10^{-3}$.
We observe that phase $\mathrm{A}$ occupies many more pores for the case $\Ca = 10^{-2}$ than for the case $\Ca = 10^{-1}$.  This is expected as the local capillary forces drive the dynamics of the flows for small values of capillary number. Finally the case $\Ca = 10^{-3}$ shows that phase $\mathrm{A}$ has invaded pores that remained filled with phase $\mathrm{B}$ in the case $\Ca = 10^{-2}$.  The simulation for $\Ca = 10^{-3}$ also exhibits many more examples of snap-off phenomena, namely break-up of one drop of phase $\mathrm{A}$ into several droplets, as phase $\mathrm{A}$ passes through pores and throats.
\begin{figure}[ht!]
\centering
\includegraphics[width=\textwidth]{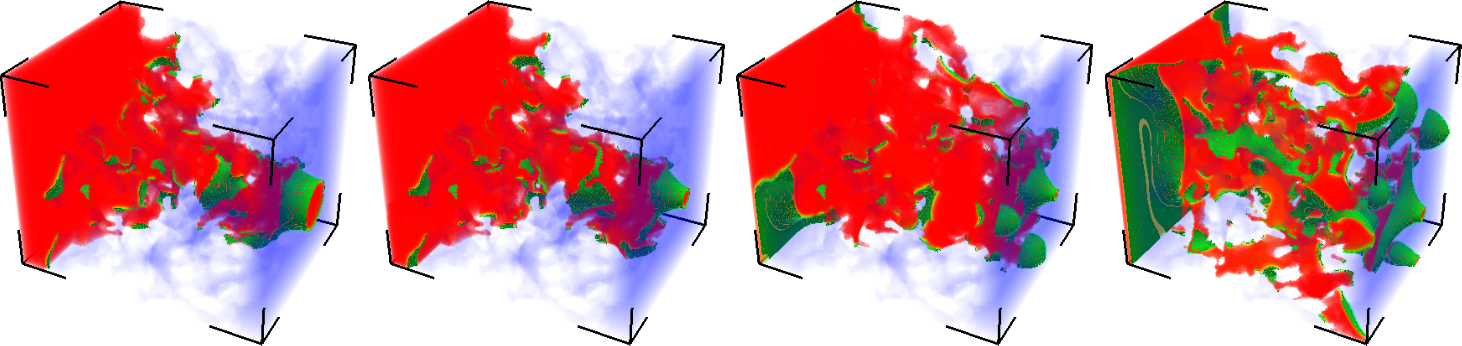}
\caption{3D snapshots of the order parameter field after injection of $10$ PV of phase $\mathrm{A}$ (in red), for different capillary numbers:
from left to right, $\Ca = 1, 10^{-1}, 10^{-2}, 10^{-3}$.  The center of the diffuse interface ($c_h=0$) is colored in green.}
\label{fig:numerical_experiments:berea_2}
\end{figure}
\par
One important application of pore scale flows is the notion of capillary desaturation, which results from the 
mobilization of the oil by increasing the magnitude of viscous forces or decreasing the magnitude of the 
capillary forces \cite{Lake1989,Hilfer2015}.
We have obtained the capillary desaturation curve with our numerical scheme.
The Berea rock sample is initially filled with oil and then flooded by water. Up to $40$ PV of water has been injected into the sample, 
which is sufficient to establish the residual oil saturation, i.e. the saturation of oil that remains trapped in the sample. 
\Cref{fig:desaturation} shows the residual oil saturation obtained for different flooding numerical experiments, each corresponding
to a different capillary number $\Ca$.  We choose twelve values for the capillary number in the range $[10^{-3}, 1]$. We observe that for capillary number greater than $10^{-2}$, the residual oil saturation
is more or less constant whereas the residual oil saturation jumps to a larger value for capillary numbers less than or equal to $4\times 10^{-3}$.
The jump in the residual oil saturation has been observed in several lab and computational experiments \cite{Lake1989,alpak2019direct,Stegemeier1974,Yeganeh}.
The transition interval $[6\times 10^{-3}, 8\times 10^{-3}]$ indicates that capillary forces dominate viscous forces for $\Ca$ less than or equal to $4\times 10^{-3}$
and that viscous forces are the dominant forces for $\Ca$ greater than equal to $10^{-2}$.

There is an extensive discussion in the literature on the correct definition of the capillary number \cite{ArmstrongBerg}. 
In particular, the correspondence
between the physical microscopic capillary number (measured in lab experiments) to the computational capillary number in phase field models
is unknown.  Capillary forces are known to dominate for physical capillary numbers less than or equal to $10^{-5}$ \cite{Lake1989}.
We show with \cref{fig:desaturation} that the computational capillary number in the phase field model is $100$ times larger
and that capillary forces dominate for computational capillary number of the order $10^{-3}$.  This shows that the capillary number
used in our model is closer to the macrocospic capillary number defined in \cite{HilfeOren,ArmstrongBerg}.

\begin{figure}[ht!]
\centering
\includegraphics[width=0.6\textwidth]{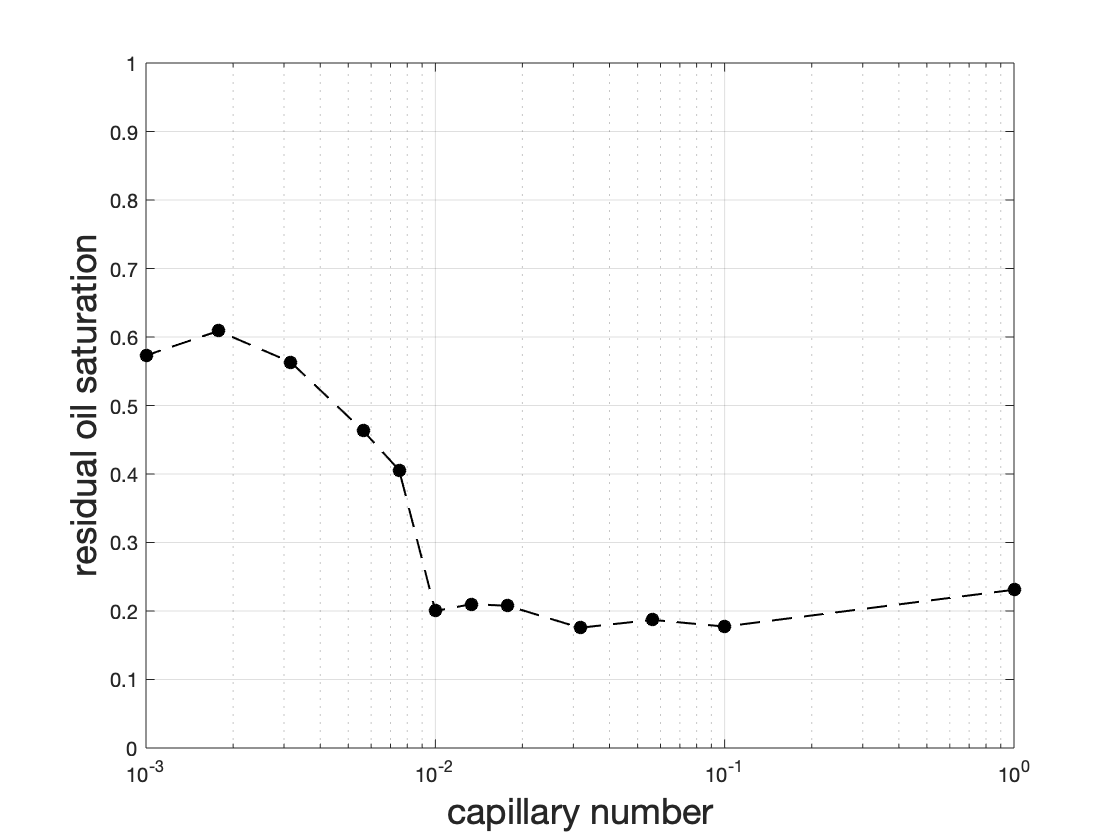} 
\caption{Desaturation curve obtained with the phase-field model. Residual oil saturation is plotted as a function of the capillary number.}
\label{fig:desaturation}
\end{figure}

\subsection{Performance and scalability}
In order to evaluate the computational performance of our simulations, we perform a strong scalability study for the pore-scale flow simulation from \Cref{sec:numerical_experiments:porous} with a capillary number of $\Ca = 10^{-2}$.
The number of degrees of freedom is about 3.8 million for both the Cahn--Hilliard equation and the elliptic equation in the pressure correction step, and it is about 11.5 million for the momentum balance equation.
We compute the first ten time steps of this simulation using one, two, four, or eight compute nodes of the Bridges cluster at the Pittsburgh Supercomputing Center~\cite{Nystrom}.
Each compute node is equipped with two Intel Xeon E5-2695 v3 CPUs ($2\times 14$ CPU cores per node) and 128~GB of memory.
\par
Our implementation is based on C++ and the Trilinos framework~\cite{Heroux2005}, which supports hybrid parallelism with MPI and OpenMP.
For our experiments, we run ten MPI processes per compute node and two OpenMP threads per MPI process, i.e., we utilize 20 of the 28 CPU cores on each compute node.
Since the node-level performance of our simulations is limited by memory bandwidth rather than by floating point operations per second (FLOPS), using all 28 CPU cores per node does not result in improved performance and can in fact harm performance.
\par
The linear systems in the Cahn--Hilliard steps are solved with a Jacobi preconditioned restarted GMRES method, and the linear systems in the momentum balance step are solved with a Jacobi-preconditioned BiCGStab method.
While the Jacobi method is a simple preconditioner, hybrid-parallel implementations of more effective preconditioners for these problems are not readily available.
We also emphasize that since the nonlinear Cahn--Hilliard equation is solved using an inexact Newton method~\cite{FLAR2018finite,Thiele2017}, much of the computational cost of the Cahn--Hilliard steps is accounted for by the linear system assembly in each Newton iteration rather than by the linear solver.
Finally, the linear systems from the pressure correction step are solved using a conjugate gradient method with a combined $p$-multigrid and algebraic multigrid (AMG) preconditioner (see, e.g.,~\cite{Helenbrook2003,Tamstorf2015}).
Specifically, a two-level $p$-multigrid method reduces the original problem, which is obtained from a piecewise linear discretization, to a problem associated with a piecewise constant discretization to which an AMG V-cycle can be applied.
Efficient preconditioning of the pressure correction systems is critical for the computational performance of our simulations.
If a simple Jacobi preconditioner is used, the solution of pressure correction systems accounts for more than 50\% of the total computational cost of large simulations~\cite{Thiele2019}.
Our $p$-multigrid implementation is available as open source software~\cite{Thiele2019cphis}.
\par
\Cref{tab:scalability} shows the results of the performance and scalability study.
We see that for this particular simulation, the Cahn--Hilliard steps accounts for the majority of the computational cost, while the cost of the Navier--Stokes steps is lower.
Computational efficiency, which we define as the ratio of observed and ideal speedup, is at least 87\% in all cases and closer to 100\% when using two or four compute nodes.
As a result, we observe a maximum speedup of $6.92$ when using eight nodes, allowing the simulation of an entire time step in $19.7$ seconds.

\begin{table}
\centering 
\caption{Computational performance and strong scalability. The table shows the computational cost of Cahn--Hilliard (CH) and Navier--Stokes (NS) updates separately and combined. Speedups and efficiencies are computed based on the total times. The times shown are the average elapsed time per simulated time step, where the average was taken over ten time steps and three repetitions of the experiment.}
\label{tab:scalability}
\begin{tabular}{rrrrrrr}
\toprule
\multicolumn{1}{c}{\textbf{\#Nodes}} & \multicolumn{1}{c}{\textbf{\#Cores}} & \multicolumn{3}{c}{\textbf{Time [s]}} & \multicolumn{1}{c}{\textbf{Speedup}} & \multicolumn{1}{c}{\textbf{Efficiency [\%]}}\\
\cline{3-5} && \multicolumn{1}{c}{\textbf{CH}} & \multicolumn{1}{c}{\textbf{NS}} & \multicolumn{1}{c}{\textbf{Total}}\\
\midrule
    1 &  20 & 90.6 & 45.9 & 136.4 & 1.00 & 100\\
    2 &  40 & 44.4 & 23.2 &  67.7 & 2.01 & 101\\
    4 &  80 & 23.4 & 11.9 &  35.4 & 3.85 &  96\\
    8 & 160 & 12.7 &  7.0 &  19.7 & 6.92 &  87\\
\bottomrule
\end{tabular}
\end{table}

\section{Conclusion}\label{sec:conclusion}
In this paper, we formulate an efficient numerical algorithm for modeling two-phase flows with varying phase density at the pore scale. The method employs (1) discontinuous polynomial approximations on cubic voxel sets; (2) a Poisson problem for the pressure-correction step that is suitable
for variable densities; and (3) flux and slope limiters that eliminate the overshoot and undershoot in the order parameter field.
Numerical results show that the proposed scheme conserves mass for closed systems. Simulations of open two-phase systems in micro structures
and in real rocks illustrate the robustness of the method.  The desaturation curve obtained with the Berea sandstone simulations indicates that capillary forces dominate the viscous forces for small values of capillary numbers.
Finally strong scalability
results show a very high computational efficiency of our method on hybrid compute nodes.  

\section*{Acknowledgments}
The authors thank Dr. Hengjie Wang (University of California, Irvine) and Dr. Florian Frank (Friedrich-Alexander-Universit\"at Erlangen-N\"urnberg) for helpful discussions. This work used the Extreme Science and Engineering Discovery Environment (XSEDE), which is supported by grants TG-DMS 190021. Specifically, it used the Bridges system, which is supported by NSF award number ACI-1445606, at the Pittsburgh Supercomputing Center (PSC).

\printbibliography

\end{document}